\documentclass[envcountsame]{llncs}
\usepackage{llncsdoc}

\usepackage{amscd,amssymb,amsmath,latexsym}

\usepackage{amsfonts}
\usepackage{graphicx}%
\usepackage[active]{srcltx}
\newtheorem{assu}[definition]{Assumption}
\newtheorem{algorithm}[definition]{Algorithm}

\newenvironment{assumption}%
{\begin{assu}\rm}%
{\end{assu} }

\newcommand{\bfm}{{\boldsymbol{m}}}

\newcommand{\bfr}{{\boldsymbol{r}}}

\newcommand{\bfv}{{\boldsymbol{v}}}
\newcommand{\bfw}{{\boldsymbol{w}}}

\newcommand{\bfz}{{\boldsymbol{z}}}

\newcommand{\bfF}{{\boldsymbol{F}}}

\newcommand{\bfG}{{\boldsymbol{G}}}

\newcommand{\bfU}{{\boldsymbol{U}}}
\newcommand{\bfV}{{\boldsymbol{V}}}

\newcommand{\bfW}{{\boldsymbol{W}}}

\def\K{{\mathbb K}}
\newcommand{\ovK}{{\overline{\K}}}

\def\I{{ I}}

\def\J{{\mathcal J}}

\def\U{{\mathcal U}}

\def\C{{\mathbb C}}

\def\Q{{\mathbb Q}}

\def\Z{{\mathbb Z}}

\newcommand{\vv}{{\bf v}}

\newcommand{\ww}{{\bf w}}

\begin{document}
%\conferenceinfo{SNC 2014}{July 28-31, 2014, Shanghai, China}

\title{Global Newton Iteration over Archimedean and non-Archimedean Fields\\ {\small  (Full Version)}}
\titlerunning{A Note on Global Newton Iteration}

%\numberofauthors{3} 
\author{Jonathan D. Hauenstein\inst{1}\thanks{Research was partly supported by NSF grant DMS-1262428 and DARPA Young Faculty Award.} \and Victor Pan\inst{2}\thanks{Research was partly supported by NSF grant CCF-1116736.} \and Agnes Szanto \inst{1}\thanks{Research was partly supported by NSF grant CCF-1217557.}}

\institute{ %Department of Mathematics,\\
 North Carolina State University,\\
       %Campus Box 8205\\
   %Raleigh, NC, 27965, USA.\\
    \email{\{hauenstein,aszanto\}@ncsu.edu}
   \and 
   %Department of Mathematics and Computer Science,\\  
Lehman College - City University of New York,\\
 %Bronx, New York 10468, USA\\
 \email{victor.pan@lehman.cuny.edu}
 }

\maketitle
% \pagestyle{myheadings}
%\markright{J. Hauenstein, V. Pan, A. Szanto }

\begin{abstract} In this paper, we study iterative methods on the coefficients of the   {\em rational univariate representation (RUR)} of a given algebraic set, called {\em global Newton iteration}. We compare two natural approaches to define locally quadratically convergent iterations: the first  one involves Newton iteration applied to the approximate roots individually and then interpolation to find the RUR of these approximate roots; the second one considers the coefficients  in the exact RUR as  zeroes of a high dimensional map defined by polynomial reduction, and applies Newton iteration  on this map. We prove that over fields with a p-adic valuation  these two approaches give the same iteration function, but over fields equipped with the usual Archimedean absolute value, they are not equivalent. In the latter case, we give explicitly the iteration function for both approaches.  Finally, we analyze the parallel complexity of the different versions of the global Newton iteration, compare them, and  demonstrate that they can be efficiently computed. The motivation for this study comes from the certification of approximate roots of overdetermined and singular polynomial systems via the recovery of an exact RUR from approximate numerical data.   \end{abstract}

%\date{\today}

%\maketitle

\section{Introduction}

Let $F_1, \ldots, F_n \in \K[x_1, \ldots, x_n]$ be  polynomials  with coefficients from a field $\K$, $\J:=\langle F_1,\ldots, F_n\rangle$ the ideal they generate, and assume that $\J$  is zero dimensional and radical. 
 We consider two cases for the coefficient field $\K$: 
\begin{description}
\item[Non-Archimedean case:] Let $R$ be a principal ideal domain,  $\K$  its field of fractions, and $p$ an irreducible element in $R$. Then, we can equip $\K$  with the $p$-adic valuation, which defines a non-Archimedean metric on vector spaces over  $\K$. The main examples include 
$R=\Z$ and $p$ a prime number, or $R=\Q[t]$  and $p=t$. Note that after clearing denominators, we can assume that $F_1, \ldots, F_m\in  R[x_1, \ldots, x_n]$. 

\item[Archimedean case:] In this case, $\K$ is a subfield of $\C$ and it is equipped with the usual absolute value. The usual Euclidean norm  defines an Archimedean metric on  vector spaces over $\K$. 
\end{description}

The objective of this paper is to  study  iterative methods on the coefficients of  the {\em rational univariate representation (RUR)} of a component of $\J$, and compare them in the Archimedean and the non-Archimedean cases. The RUR  of a component of $\J$, originally defined in \cite{Rou99}, is a simple representation of a subset of the common roots of $F_1,\ldots, F_n$, expressing the coordinates of these common roots as Lagrange interpolants at  nodes which are given as the roots of a univariate polynomial   (see definition below).  

We study two natural approaches for iterations  that are locally quadratically convergent to an exact RUR of a component of $\J$, both based on Newton's method:
\begin{itemize}
\item To update an RUR,  apply the usual $n\times n$ Newton iteration to each common root of the old RUR, and compute the updated RUR which defines  these updated roots. In this approach we assume inductively that the common roots of the iterated RUR's  are known exactly. 

\item Consider the map that takes an RUR and returns the reduced form of the input polynomials  $F_1,\ldots, F_n$ modulo the RUR. Since an exact RUR of a component of $\J$ is a zero of this map, we apply Newton's method to this map. 
\end{itemize}
Note that in the $p$-adic case, the first iteration was studied in \cite{Grobner-free} where they gave the  iteration function explicitly and analyzed its complexity in terms of straight-line programs, while the second approach was proposed in \cite{Trinks85}, without giving the iteration explicitly. 

The main results of this paper are as follows: First we prove that  the above two approaches give the same iteration function in the p-adic valuation. Next, we show that in the Archimedean case the two iterations are not equivalent. In this case, we give the explicit iteration functions for both approaches and show that they are  also different from the iteration function presented 
in~\cite{Grobner-free} and interpreted for the Archimedean case. We illustrate the methods on an example involving the mobility of a spacial mechanism. Finally, we analyze the parallel complexity of both approaches in the Archimedean case: For the first approach, we use $n\times n$ Newton iterations independently for each root and an efficient parallel Vandermonde linear solver for  Lagrange interpolation. For the second approach we utilize efficient parallel  Toeplitz-like linear system solvers to compute modular inverses of univariate polynomials. We present a version of  the algorithms of \cite{Pan1992} to solve Toeplitz-like linear systems that uses a more efficient displacement representation with factor circulant matrices defined in \cite[Example 4.4.2]{Pan2001} rather than triangular Toeplitz matrices. Finally, we briefly discuss 
the computation of modular inverses  and cofactors
when all the roots of at least one of the two
associated input polynomials are simple and known.

\subsection{Related work}

The motivation to study numerical approximations of  RUR's   come from a  work in progress in  \cite{AkogHausSza13} to certify approximate roots of overdetermined and singular polynomial systems over $\Q$. For well-constrained non-singular systems, Smale's $\alpha$-theory (see \cite[Chapters 8 and 14]{BlCuShSm})  gives a tool for the certification of approximate roots,  as was explored and implemented in  {\bf alphaCertified} \cite{Hauenstein-Sotille}. However, {\bf alphaCertified} does not straightforwardly  extend to overdetermined or singular systems: in  \cite{Hauenstein-Sotille}, they propose to use universal lower bounds for the minimum of positive polynomials on a disk, such as in  \cite{Jeronimo-Perrucci}, but they conclude that  such bounds are ``too small to be practical.'' To overcome this difficulty, in  \cite{AkogHausSza13} it is proposed to iteratively compute the exact RUR of a rational component from approximations of the roots, and then use   the machinery of \cite{Hauenstein-Sotille} to certify approximate roots of this RUR. 
While \cite{AkogHausSza13} is devoted to considerations about the global behavior of the iteration,  this paper considers different choices of the iteration function and their parallel complexity. 

The iterative algorithms that are in the core of this paper  are the Archimedean adaptations of what is known as ``global Newton iteration" or ``multivariate Hensel lifting" or ``Newton-Hensel lifting" in the computer algebra literature, where it is defined for the non-Archimedean case.   Various versions of the Newton-Hensel lifting were applied in many applications within computer algebra, including in  univariate polynomial factorization 
\mbox{\cite{Zassenhaus69,LLL82}}, multivariate polynomial factorization \cite{Chistov84,Grigorev84,Kaltofen85}, gcd of sparse multivariate polynomials \cite{Kaltofen85b},  lexicographic    and general Gr\"obner basis  computation of zero dimensional algebraic sets  \cite{Trinks85,Winkler88}, geometric resolution  of equi-dimensional algebraic sets \cite{GHHMPM97,GHMM98,HKPSW2000,Grobner-free}, Chow forms \cite{JKSS04}, and sparse interpolation \cite{AveKrickPac2006}. As we mentioned above, the most related to this paper are  \cite{Trinks85,Grobner-free}.

Computing numerical approximation to symbolic objects  in the Archimedean metric is not new either. There is a significant literature studying such hybrid symbolic-numeric algorithms, and without trying to give a complete bibliography,  the following mentions the papers that are the closest to our work.

Closest to our approach is the literature on  finding the vanishing ideal of a finite point set given with limited precision. In \cite{CaPaHaMo2001}, they give an  algorithm that given  {\em one} approximate zero of a polynomial system, finds  the  RUR of the irreducible component  containing the corresponding exact roots in randomized polynomial time. The algorithm in \cite{CaPaHaMo2001} applies the univariate results of  \cite{KanLenLov1988} using  lattice basis reduction.   The main point of our approach in this paper and in \cite{AkogHausSza13} is that we assume to know {\em all} approximate roots of a rational  component, so in this case we can compute the exact RUR much more efficiently, and in parallel.

The papers \cite{CaPaHaMo2001,Stetter2004,MoTr2008,AFT2008,HKPP2009,Fassino2010,FasTor2013} use a more general approach than the one here, by computing  border bases for a given set of approximate roots, which avoids defining a random primitive element as is done for  RURs used in this paper. For general polynomial systems, numerical computation of Gr\"obner bases was proposed, for example, in \cite{Shir93,Shir96,Lich2010,TraZa2002}.  The focus of these papers is to find numerically stable support for the bases, which we assume to be given here by the primitive element.

\section{Preliminaries}

Let us start with recalling the notion of Roullier's {\em Rational Univariate Representation (RUR)}, originally defined in \cite{Rou99}. Instead of defining the RUR of an ideal $\J$, here we only define the notion of the {\em RUR of a component of $\J$}, which is a weaker notion.  We follow the  notation in \cite{AkogHausSza13}.  

Let $\K$ be a field.  Given ${\bf F}=(F_1, \ldots, F_n)\subset \K[x_1, \ldots, x_n]$ for some $n$, and assume that  the ideal  $\J:=\langle F_1, \ldots, F_n\rangle$ is radical and zero dimensional. Then the factor ring $ \K[x_1, \ldots, x_n]/ \J$ is a finite dimensional vector space over $\K$, and we denote 
$$
\delta:=\dim_\K \K[x_1, \ldots, x_n]/\J
.$$
 Furthermore, for almost all $(\lambda_1, \ldots, \lambda_n)\in \K^n$ (except a Zariski closed subset), the linear combination  
$$
u(x_1, \ldots, x_n):=\lambda_1 x_1 + \cdots + \lambda_nx_n
$$
is a {\em primitive element} of $\J$, i.e. the powers $1, u, u^2, \ldots, u^{\delta-1}$ form a linear basis for $\K[x_1, \ldots, x_n]/ \J$ (c.f. \cite{Rou99}). 

In the algorithms that follow, we compute an RUR that may not generate the  ideal $\J$, nevertheless 
the polynomials $F_1, \ldots, F_n$ vanish modulo the RUR. In this case the RUR will generate a {\em component of $\J$}. 
We have the following definition:

\begin{definition}  Let $\J=\langle F_1, \ldots, F_n\rangle\subset \K[x_1, \ldots, x_n]$ be as above. Let $\lambda_1 x_1 + \cdots + \lambda_nx_n$ be a primitive element of $\J$. 
We call the polynomials  
\begin{equation}\label{RUR}
T-\lambda_1 x_1 + \cdots + \lambda_nx_n, \;q(T),\; v_1(T), \ldots, v_n(T)
\end{equation}
 a {\em Rational Univariate Representation (RUR) of a component of $\J$} if it satisfies the following properties:
\begin{itemize}
\item $q(T)\in \K[T]$ is a monic polynomial of degree $d\leq \delta$,
\item ${\rm gcd}_T(q(T), q'(T))=1$ where $q'(T)=\frac{\partial q(T)}{\partial T}$,
\item $v_1(T), \ldots, v_n(T)\in \K[T]$ are all degree at most $d-1$ and satisfy
$$
\lambda_1 v_1(T) + \cdots + \lambda_nv_n(T)=T ,
$$
\item for all $i=1, \ldots, n$ we have 
$$ F_i(v_1(T), \ldots, v_n(T)) \equiv 0 \mod q(T).
$$
\end{itemize}
\end{definition}

First note that  the set
\begin{eqnarray}\label{GB}
\{ q(T), \; x_1-v_1(T), \ldots, x_n-v_n(T)\}
\end{eqnarray} 
forms  a Gr\"obner basis for the ideal it generates with respect to the lexicographic monomial ordering with $T<x_1<\cdots < x_n$.

 Next 
let us recall the relationship between the RUR of  a component of $\J$ and its (exact) roots.   Let $$V(\J)= \{\xi_1, \ldots, \xi_\delta\} \subset \C^n$$ be the set of  common roots of $\J$.  Denote $\xi_i=(\xi_{i,1}, \ldots, \xi_{i,n})$ for $i=1, \ldots, \delta$. Then for any $n$-tuple $(\lambda_1, \ldots \lambda_n)\in \K^n$ such that for $i,j=1, \ldots, \delta$
$$
\lambda_1 \xi_{i,1} + \cdots + \lambda_n\xi_{i,n}\neq \lambda_1 \xi_{j,1} + \cdots + \lambda_n\xi_{j,n} \quad \text{ if } i\neq j,
\; $$
we can define the primitive element $u=\lambda_1 x_1 + \cdots + \lambda_nx_n$ for $\J$. Since all roots of $\J$ are distinct, such primitive element exist, and can be computed from the roots $\{\xi_1, \ldots, \xi_\delta\}$, or using randomization. Fix such $(\lambda_1, \ldots \lambda_n)\in \K^n$.
 For $d\leq \delta$ let $\{\xi_1, \ldots, \xi_d\}$ be  a subset of  $V(\J)$ 
 and define 
\begin{eqnarray}\label{mui}
\mu_i:= \lambda_1 \xi_{i,1} + \cdots + \lambda_n\xi_{i,n}, \quad i=1, \ldots d.
\end{eqnarray}
The RUR of the component of $\J$ corresponding to the subset $\{\xi_1, \ldots, \xi_d\}\subset V(\J)$ is defined by 
\begin{eqnarray}\label{qT}
q(T):= \prod_{i=1}^d (T-\mu_i), 
\end{eqnarray} and  for each $j=1, \ldots, n,$ the polynomial $v_j(T)$   is the unique Lagrange interpolant of degree at most $d-1$ satisfying
\begin{eqnarray}\label{vj}
v_j(\mu_i) = \xi_{i,j} \quad \text{ for } i=1, \ldots, d.
\end{eqnarray}
Note that if $\J$ is defined by polynomials over $\K$, then the polynomials in the RUR of $\J$ have coefficients in $\K$, but that is not true for all components  of $V(\J)$. We call a subset  $\{\xi_1, \ldots, \xi_d\}\subset V(\J)$ a {\em rational component of $\J$} if the corresponding RUR has also coefficients in $\K$.

\section{Global Newton Iteration} 

In this section we describe  iterative methods that improves the accuracy of the coefficients of the RUR of a component of $\J$.  We use a similar approach as in \cite[Section 4]{Grobner-free}, but instead of a coefficient ring with the p-adic absolute value, here we make adaptations to  coefficient field $\K\subseteq \C$ equipped with the usual absolute value. We start with recalling the definitions given in  \cite[Section 4]{Grobner-free}.

\subsection{Non-Archimedean Global Newton iteration}
 First, we briefly describe the global Newton iteration defined  in \cite[Section 4]{Grobner-free}. There the coefficient domain is the ring $\Q[t]$ and the non-Archimedean metric is defined by the irreducible element  $ t\in \Q[t]$.
 They consider a square system  $\bfF=(F_1, \ldots, F_n)$ with $F_i\in \Q[t][x_1, \ldots, x_n]$. Let 
 $$u(x_1, \ldots, x_n)=\lambda_1x_1+\cdots + \lambda_nx_n=T$$ be a random primitive element for $\langle F_1, \ldots, F_n\rangle $.  Furthermore,  define
$$
\I:=\langle t^k\rangle \text{ for some } k.
$$

 In  \cite[Section 4]{Grobner-free} they assume that some initial approximate RUR is given 
for a component of $\J$ :
$$
 \;q(T), \;{\bf v}(T):=(v_1(T), \ldots, v_n(T)) \in \Q[t][T]
$$
satisfying
the  following assumptions:
 
\begin{assumption}\label{modassu} Let $\bfF$, $u$, $\I$, $q(T)$ and $\vv(T)$ be  as above. Then
\begin{enumerate}
%\item $\frac{\partial q(T)}{\partial T}$ is invertible modulo $q(T)$ and $\I$,
\item $q(T)$ is monic and has degree $d$, 
\item $v_i(T)$ has degree at most $d-1$,
\item $\bfF(\vv(T))\equiv 0 \mod q(T) \mod \I,$
\item $\lambda_1v_1(T)+\cdots +\lambda_n v_n(T) = T\mod \I$,
\item   $J_\bfF(\vv(T))  := \left[\frac{\partial F_i}{\partial x_j}(\vv(T))\right]_{i,j=1}^n$  is invertible modulo $q(T)$ and $\I$.
\end{enumerate}
\end{assumption}

They define the following updates: 
\begin{definition}\label{constr1} 
Assume that  $\bfF$, $u$, $q(T)$, $\vv(T)$ and $\I$ satisfy Assumption \ref{modassu}.  Then in \cite[Section 4]{Grobner-free} they  define 
{\small \begin{eqnarray}
u&=&\sum_{i=1}^n\lambda_i x_i=T \text{ remains the same as for the initial RUR, }\label{uu}\nonumber\\
\bfw(T) &:=&
%\hspace{-.4cm} \left[\begin{array}{c}w_1(T)\\\vdots\\w_n(T)\end{array}\right]:=
\vv(T)-   \left(J_\bfF (\vv(T))^{-1}\bfF(\vv(T))\hspace{-.3cm} \mod q(T)\right)\hspace{-.3cm} \mod \I^2,  \label{ww}\nonumber\\
\label{Delta}
\Delta(T)&:=&\sum_{i=1}^n \lambda_i w_i(T) -T \mod \I^2,\nonumber\\
\label{VV}
\bfV(T)&:=&
% \left[\begin{array}{c}V_1(T)\\\vdots\\V_n(T)\end{array}\right]:= 
\bfw(T)-\left(\Delta(T)\cdot\frac{\partial \bfw(T)}{\partial T} \mod q(T)\right)\hspace{-.3cm} \mod \I^2,\nonumber\\
Q(T)&:=&q(T)-\left(\Delta(T)\cdot\frac{\partial q(T)}{\partial T} \mod q(T)\right) \mod \I^2. \label{Q}\nonumber
\end{eqnarray}}
\end{definition}
In \cite[Section 4]{Grobner-free} they prove the following:

\begin{proposition}[\cite{Grobner-free}] \label{prop1} Assume that  $\bfF$, $u$, $q(T)$, $\vv(T)$ and $\I$ satisfy Assumption \ref{modassu} and let  $\bfw(T)$, $\Delta(T)$, $\bfV(T)$, $Q(T)$ be as in Definition \ref{constr1}. Then 
\begin{enumerate}
\item[(i)] $\vv(T)\equiv \bfw(T)\equiv \bfV(T)$, $q(T)\equiv Q(T)$  and $\Delta(T)\equiv 0 \mod \I$
\item[(ii)] $\bfF(\bfw(T))\equiv 0  \mod q(T) \mod \I^2,$
\item[(iii)] { $\langle q(T), U-T-\Delta(T), x_1-w_1(T), \ldots, x_n-w_n(T)\rangle = \langle Q(U), T-U-\Delta(U), x_1-V_1(U), \ldots, x_n-V_n(U)\rangle \mod \I^2,
$}
\item[(iv)] $\bfF(\bfV(T))\equiv 0  \mod Q(T) \mod \I^2,$
\item[(v)] $\lambda_1V_1(T)+\cdots +\lambda_n V_n(T) = T\mod \I^2$.
\end{enumerate}
\end{proposition}

\subsection{First Construction}
Our first variation of Definition \ref{constr1} will have the property that it agrees to the approximate RUR obtained from the approximate roots via Lagrange interpolation as was described in the Preliminaries. We give our definition over some general coefficient ring $R$, but later we will use $R=\K \subset \C$,   or $\Q[t]/\I^2$. We need the following assumptions:

\begin{assumption}\label{qassu} Let $\bfF=(F_1, \ldots, F_n)$, $u=\lambda_1x_1+\cdots+\lambda_n x_n$, $q(T)$ and $\vv(T)=(v_1(T),\ldots , v_n(T))$ polynomials over some Euclidean domain $R$  as above. We assume that 
\begin{enumerate}
\item $q(T)$ is monic and has degree $d$, 
\item $v_i(T)$ has degree at most $d-1$,
\item $\frac{\partial q(T)}{\partial T}$ is invertible modulo $q(T)$,
\item  $\lambda_1v_1(T)+\cdots +\lambda_n v_n(T) = T$,
\item    $J_\bfF(\vv(T))  := \left[\frac{\partial F_i}{\partial x_j}(\vv(T))\right]_{i,j=1}^n$  is invertible modulo $q(T)$.
\end{enumerate}
\end{assumption}

Our first construction for the update is defined as follows:

\begin{definition}\label{constr2} Assume that  $\bfF$, $u$, $q(T)$ and $\vv(T)$ satisfy Assumption \ref{qassu}. Then we define
{\small \begin{eqnarray}
&&u=\sum_{i=1}^n\lambda_i x_i=T \nonumber\\
&&\bfw(T):=\vv(T)-   \left(J_\bfF (\vv(T))^{-1}\bfF(\vv(T))\mod q(T)\right), \nonumber\\
&&\Delta(T):=\sum_{i=1}^n \lambda_i w_i(T) -T  \text{ so far the same as in Definition \ref{constr1}},\\
\label{tVV}&&\tilde{\bfV}(T+\Delta(T))= \left[\begin{array}{c}\tilde{V}_1(T+\Delta(T))\\\vdots\\\tilde{V}_n(T+\Delta(T))\end{array}\right]:= \bfw(T)\mod q(T),\\
&&\Delta\tilde{Q}(T+\Delta(T)):=-(T+\Delta(T))^d \mod q(T)  \label{DtQ}\\
&& \tilde{Q}(T):=\Delta\tilde{Q}(T) + T^d  \label{tQ}
\end{eqnarray}}
\end{definition}
Note that in Definition \ref{constr2} we define  $\tilde{\bfV}(T+\Delta(T))$ and not   $\tilde{\bfV}(T)$, but the coefficients of $\tilde{V}_i(T)$ can be obtained as  solutions of  linear systems. Similarly for the coefficients of $\Delta\tilde{Q}(T)$. In the next proposition we examine  these linear systems and give conditions on the existence and uniqueness of their solutions. 

\begin{proposition} \label{unique}  The coefficients of the polynomials $\tilde{V}_i(T)$ in (\ref{tVV}) for $i=1, \ldots, n$, and the coefficients of the polynomial $\Delta\tilde{Q}(T)=\tilde{Q}(T)-T^d$ in (\ref{DtQ}) are the solutions of  $d\times d$ linear systems  with a common coefficient matrix that has columns  which are 
 the coefficient vectors of 
$$(T+\Delta(T))^j \mod q(T)\quad  \text{ for } j=0, \ldots, d-1.
$$
This coefficient matrix is non-singular if and only if  $u$ is   a primitive element for the ideal $\langle q(T), x_1-w_1(T), \ldots, x_n-w_n(T)\rangle. $
\end{proposition}

\begin{proof}
A closer look of the definition in (\ref{tVV})  and (\ref{DtQ}) gives the coefficient matrix of the linear systems defining $\tilde{V}_i(T)$ and $\Delta\tilde{Q}(T)$ as stated, with a common coefficient matrix. This   linear system  has unique solution if and only if $\;1, \,T+\Delta(T), \ldots, (T+\Delta(T))^{d-1}$ are linearly independent modulo $q(T)$, and since 
\begin{equation}\label{uww}
T+\Delta(T)=\sum_{I=1}^n \lambda_iw_i(T)=u(\ww(T)),
\end{equation}
this is equivalent to $u$ being a primitive element of the updated system. \end{proof}

The following proposition compares Definitions \ref{constr1} and \ref{constr2} in cases when  the coefficient ring is  $R=\Q[t]/\I^2$. 
\begin{proposition} Assume that the conditions of Assumption \ref{modassu} are satisfied and that $u$ is a primitive element for $\langle q(T), x_1-w_1(T), \ldots, x_n-w_n(T)\rangle$ as in Proposition \ref{unique}. 
 Let $\bfV(T)$ and $Q(T)$ be as in Definition \ref{constr1} and $\tilde{\bfV}(T)$ and $\tilde{Q}(T)$ be as in Definition \ref{constr2} for $R=\Q[t]/\I^2$.  Then ${\bfV}(T)\equiv \tilde{\bfV}(T)$ and $ Q(T)\equiv\tilde{Q}(T)$ mod $\I^2$. 
\end{proposition}

\begin{proof}
From Proposition \ref{prop1}.(iii) we get $
V_i(T+\Delta(T))\equiv w_i(T) \text{ and } Q(T+\Delta(T))\equiv 0 \mod q(T)\mod \I^2.
$ 
Since $Q(T)$ is a monic polynomial of degree $d$,  the coefficients of its  degree $\leq d-1$ terms are uniquely determined modulo $q(T)$, so  
we have that 
 $$
Q(T+\Delta(T))-(T+\Delta(T))^d\equiv -(T+\Delta(T))^d\mod q(T) \mod \I^2.
$$
Since $\tilde{\bfV}$ and $\tilde{Q}$ are uniquely defined by these properties, they must be equal to $\bfV$ and $Q$ respectively. \end{proof}

The following proposition connects the approximate RUR defined in Definition \ref{constr2} to the ones obtained by applying one step of Newton iteration on the  approximate roots, as promised in the Introduction: 
\begin{proposition}\label{roots}
Let $\bfF=(F_1, \ldots, F_n)\subset \K[x_1, \ldots, x_n]$ as above. Assume that the polynomials  
$
u=\lambda_1x_1+\cdots +\lambda_nx_n,  \;q(T), \;{\bf v}(T):=(v_1(T), \ldots, v_n(T))
$
satisfy Assumption \ref{qassu}. Let $\bfz_1, \ldots, \bfz_d\in \ovK^n$ be the exact roots of 
$
\langle 
q(T), x_1-v_1(T), \ldots, x_n-v_n(T)\rangle \cap\K[x_1, \ldots, x_n]$ where $\ovK$ is the algebraic closure of $\K$. 
 Let 
$$
\tilde{\bfz}_i:= \bfz_i- J_\bfF (\bfz_i)^{-1}\bfF(\bfz_i) \quad i=1, \ldots, d 
$$
be one step of Newton iteration. Assume that $u(\tilde{\bfz}_i)\neq u(\tilde{\bfz}_j)$ for $i\neq j$. Then $\tilde{Q}(T), \tilde{V}_1(T), \ldots, \tilde{V}_n(T)$ defined in Definition \ref{constr2} is the exact RUR of $\{\tilde{\bfz}_1, \ldots, \tilde{\bfz}_d\}$, with  
$
\sum_{i=1}^n \lambda_i\tilde{V}_i(T) =T.
$
  \end{proposition}

\begin{proof} 
Using the notation $\bfz_i=(z_{i,1}, \ldots, z_{i,n})$  and $\tilde{\bfz}_i=(\tilde{z}_{i,1}, \ldots, \tilde{z}_{i,n})$ we define 
$$
\mu_i := u({\bfz}_i)=\sum_{j=1}^n\lambda_j z_{i,j} \text{ and } \tilde{\mu}_i := u(\tilde{\bfz}_i)=\sum_{j=1}^n\lambda_j\tilde{z}_{i,j} \;\;i=1, \ldots, d.
$$
Then  for all $i=1, \ldots, d$ and $j=1, \ldots, n$ we have 
$
q(\mu_i)=0, \; w_j(\mu_i) = \tilde{z}_{i,j} \text{ and } \Delta(\mu_i)=\tilde{\mu_i}-\mu_i.
$
Note that $u$ is a primitive element for $\langle q(T), x_1-w_1(T), \ldots, x_n-w_n(T)\rangle$ if and only if $\tilde{\mu}_i\neq \tilde{\mu}_j$ for $i\neq j$. 
In this case for all $i=1, \ldots, d$ and $j=1, \ldots, n$  we have from (\ref{tVV}), (\ref{DtQ}) and (\ref{tQ}) that 
$$\tilde{V}_j(\tilde{\mu}_i)=w_j(\mu_i)=\tilde{z}_{i,j} \text{ and } \tilde{Q}(\tilde{\mu}_i)=\Delta\tilde{Q}(\tilde{\mu}_i)+\tilde{\mu}_i^d=0.$$
This proves that $\tilde{Q}(T), \tilde{V}_1(T), \ldots, \tilde{V}_n(T)$ is the exact RUR of $\{\tilde{\bfz}_1, \ldots, \tilde{\bfz}_d\}$. Finally, the last claim follows from
$$
\sum_{i=1}^n \lambda_i\tilde{V}_i(T+\Delta(T)) = \sum_{i=1}^n \lambda_iw_i(T)=T+\Delta(T).\quad
$$
\end{proof}

\begin{corollary} Let $R=\K\subset\C$  be a field  equipped with the usual absolute value, and consider the Euclidean norm on the coefficient vectors of polynomials over~$\C$. 
Then the  iteration defined by Definition \ref{constr2} is locally quadratically convergent to an exact RUR of a component of $\langle \bfF \rangle $ over an algebraic extension of $\K$, as long as  Assumption \ref{qassu} is satisfied in each iteration.
\end{corollary}

\subsection{Second Construction}

Our second variation of Definition \ref{constr1} will have the property that it can be interpreted as an $(n+1)d$ dimensional Newton iteration  as follows. Given $\bfF=(F_1, \ldots, F_n)$ and $u=\sum_{i=1}^n \lambda_ix_i$ in $R[x_1, \ldots, x_n]$ as before, we define the map $$\Phi:R^{(n+1)d}\rightarrow R^{(n+1)d}$$ as the map of the coefficient vectors of the following degree $d-1$ polynomials:
{\small \begin{equation}\label{Phi}
\Phi : \;\;\left[\begin{array}{c}
v_1(T) \\ \vdots\\ v_n(T)\\ \Delta q(T) 
\end{array}\right] \mapsto \left[\begin{array}{c}
F_1(\vv(T))\hspace{-.3cm}\mod q(T)\\
 \vdots\\
F_n(\vv(T))\hspace{-.3cm}\mod q(T)\\
\sum_{i=1}^n \lambda_i v_i(T) -T
  \end{array}\right],
\end{equation}}
where 
$$
 \Delta q(T):= q(T)-T^d.
 $$
If $u, q(T), v_1(T), \ldots, v_n(T)$ is an exact RUR of a component of $\langle \bfF\rangle$ then
$$
\Phi\left(v_1(T), \ldots, v_n(T), \Delta q(T) \right)=0.
$$
So one can apply the $(n+1)d$ dimensional Newton iteration  to locally converge  to the coefficient vector  of an exact RUR which is a zero of $\Phi$. Note that below we will consider the map $\Phi$ as a map between 
$$
\Phi:\left( R[T]/\langle q(T) \rangle \right)^{n+1} \rightarrow \left( R[T]/\langle q(T) \rangle \right)^{n+1}, 
$$
and note that $\left( R[T]/\langle q(T) \rangle \right)^{n+1}$ and $R^{(n+1)d}$ are isomorphic as vectors spaces  when $R=\K$ a field. Moreover, as we will see below,  the Newton iteration for $\Phi$ respects the algebra structure of  $\left( R[T]/\langle q(T) \rangle \right)^{n+1}$ as well.

The first lemma gives the Jacobian matrix of $\Phi$.

\begin{lemma}\label{PhiLemma} Let $\bfF=(F_1, \ldots, F_n)$, u, $q(T)$, $\vv(T)$ and $\Phi$ be as above.  For $i=1, \ldots, n$ define $m_i(T)$ and $r_i(T)$ as the quotient  and remainder in the division with remainder:
\begin{equation}\label{rm}
F_i(\vv(T))=m_i(T)q(T) + r_i(T).
\end{equation}
Then the Jacobian matrix of $\Phi$ defined in (\ref{Phi}) and considered as a map on $
\left(R[T]/\langle q(T) \rangle \right)^{n+1}$  is given by
{\small \begin{equation}\label{JacobPhi}
J_\Phi( \vv(T) ,\Delta q(T))
:=
\begin{array}{|ccc|c|c}
\multicolumn{3}{c}{\scriptsize{n}}&\multicolumn{1}{c}{\scriptsize{1}}\\
\cline{1-4}
      & & & -m_1(T)&\\
      & J_\bfF(\vv(T)) & &\vdots& n\\
      & & & -m_n(T)&\\
      \cline{1-4}
   \lambda_1& \cdots &\lambda_n & 0 & 1  \\
\cline{1-4} \multicolumn{2}{c}{}
\end{array}\mod q(T).
\end{equation}}
\end{lemma}

\begin{proof}
Using the notation $\bfr=(r_1(T), \ldots, r_n(T))$ from (\ref{rm}), we have that 
$$
\Phi( \vv(T), \Delta q(T))= (\bfr(T), \sum_{i=1}^n \lambda_i v_i(T)).
$$
Let $v_{i,j}$ be the coefficient of $T^j$ in $v_i(T)$ for $i=1, \ldots, n$ and $j=0, \ldots, d-1$. 
Then for $k=1, \ldots, n$ 
\begin{eqnarray*}
\frac{\partial \,r_k(T)}{\partial v_{i,j}}&=& \frac{\partial F_k(\vv(T))}{\partial v_{i,j}}-q(T) \frac{\partial m_k(T)}{\partial v_{i,j}} \\
&&= \frac{\partial F_k(\vv(T))}{\partial v_{i,j}} \mod q(T) \\
&&= \frac{\partial F_k}{\partial x_i}(\vv(T))\cdot \frac{\partial v_i(T))}{\partial v_{i,j}}\mod q(T) \\
&&= \frac{\partial F_k}{\partial x_i}(\vv(T)) \cdot T^j \mod q(T).
\end{eqnarray*}
This shows that the block  corresponding to the derivatives of $v_{i,j}$ for $j=0, \ldots, d-1$ defines modular multiplication by the polynomials $\frac{\partial F_k}{\partial x_i} (\vv(T))$, which is an entry of $J_\bfF(\vv(T))$. 
Furthermore, 
$$
\frac{\partial\left(\sum_{k=1}^n \lambda_k v_k(T)\right)}{\partial v_{i,j}}=\lambda_iT^j
$$ which gives the last row of (\ref{JacobPhi}). 
Let $q_j$ be he coefficient of $T^j$ in $\Delta q(T)=q(T)-T^d$ for  $j=0, \ldots, d-1$. Then
\begin{eqnarray*}
\frac{\partial r_k(T)}{\partial q_{j}}&=& \frac{\partial F_k(\vv(T))}{\partial q_{j}}-q(T) \frac{\partial m_k(T)}{\partial q_{j}}- m_k(T) \frac{\partial q(T)}{\partial q_{j}} \\
&&= - m_k(T) \frac{\partial q(T)}{\partial q_{j}} \mod q(T)\\
&& = -m_k(T) T^j  \mod q(T)
\end{eqnarray*}
and $\frac{\partial \lambda_k v_k(T)}{\partial q_{j}}=0$ which gives the last column.
\end{proof}

\begin{remark}
The proof of Lemma \ref{PhiLemma} shows that the Jacobian matrix of $\Phi$ considered as a map on $R^{(n+1)d}$ is the $(n+1)d\times (n+1)d$ matrix obtained from $J_\Phi$ in (\ref{JacobPhi}) by replacing every entry by the $d\times d$ matrix of multiplication modulo $q(T)$ by the polynomial in that entry. In other words, if the polynomial in the $(i,j)$-th entry of $J_\phi$ is $p(T)$ then the $k$-th column of this $d\times d$ block is the coefficient vector of $T^{k-1}p(T) \mod q(T)$. 
\end{remark}

Next, we give explicitly the iteration function corresponding to the Newton iteration on $\Phi$, using polynomial arithmetic modulo $q(T)$. We need the following assumptions:
 
 \begin{assumption}\label{assu3} Let $\bfF(x_1, \ldots, x_n)$, $u=\sum_{i=1}^n \lambda_ix_i$, $q(T)$ and $\vv(T)$ polynomials over some Euclidean domain $R$  as above. We assume that 
{\begin{enumerate}
\item $q(T)$ is monic and has degree $d$, 
\item $v_i(T)$ has degree at most $d-1$,

\item $\frac{\partial q(T)}{\partial T}$ is invertible modulo $q(T)$,
\item  $\lambda_1v_1(T)+\cdots +\lambda_n v_n(T) = T$,
\item    $J_\bfF(\vv(T))  := \left[\frac{\partial F_i}{\partial x_j}(\vv(T))\right]_{i,j=1}^n$  is invertible modulo $q(T)$.
\item $J_\Phi:=J_\Phi(\vv(T), \Delta q(T))$ defined in (\ref{JacobPhi}) is invertible modulo $q(T)$.
\end{enumerate}}
\end{assumption}

\begin{definition}\label{constr3} Let $\bfF(x_1, \ldots, x_n)$, $u(x_1, \ldots, x_n)$, $q(T)$ and $\vv(T)$ be polynomials over $R$ satisfying Assumption \ref{assu3}. Then we define 
{\small \begin{eqnarray}
&&u=\sum_{i=1}^n\lambda_i x_i=T \nonumber\\
&&\bfw(T) :=\vv(T)-   \left(J_\bfF (\vv(T))^{-1}\bfF(\vv(T))\mod q(T)\right), \nonumber\\
&&\Delta(T):=\sum_{i=1}^n \lambda_i w_i(T) -T  \text{ same as in Definitions \ref{constr1} and \ref{constr2}},\label{ww3}\\
&& \bfr(T):= \bfF(\vv(T))\mod q(T)\label{r}\\
&& \bfU(T):=  \frac{\partial \vv(T)}{\partial T} -  \left(J_\bfF (\vv(T))^{-1}\frac{\partial \bfr(T)}{\partial T} \mod q(T)\right), \label{U}\\
&&\Lambda(T):= \sum_{i=1}^n \lambda_i U_i(T)\label{Lambda} \text{ that we will show to be invertible modulo }  q(T)\\
\label{bVV}&&\bar{\bfV}(T):= \bfw(T)- \left( \frac{\Delta(T)}{\Lambda(T)} \bfU(T)\mod q(T)\right),\\
&&\bar{Q}(T):=q(T) - \left(\frac{\Delta(T)}{\Lambda(T)}\frac{\partial q(T)}{\partial T} \mod q(T) \right).  \label{bQ}
\end{eqnarray}}
\end{definition}

\begin{remark} Note that in general derivation and modular arithmetic do not commute, i.e.
$$\frac{\partial p(T)}{\partial T}\hspace{-.3cm}\mod q(T) \neq \frac{\partial ( p(T)\hspace{-.3cm} \mod q(T))}{\partial T}.
$$
For example if $p=T^2$ and $q=T^2-1$ then the left hand side is $2T$ but the right hand side is $0$. That is why we first had to introduce the reduced form of $\bfF(\vv(T))$ modulo $q(T)$ in (\ref{r}) and then use its derivative by $T$ in (\ref{U}).

\end{remark}

\begin{remark} Note that if $R=\Q[t]$ then and $\I=\langle t^k\rangle $ for some $k\geq 1$ then  

$$
\frac{\Delta(T)}{\Lambda(T)} \equiv \Delta(T) \text{ and } \bfU(T)\equiv\frac{\partial \ww(T)}{\partial T} \mod  q(T) \mod \I^2,
$$
thus our second construction is equivalent to the one in Definition \ref{constr1}. However, when our coefficient ring $R$ is a field $\K\subset \C$,  the polynomials in $\bfU(T)$ are  not the partial derivatives of the ones in $\ww(T)$, so we get a different iteration in Definition \ref{constr3} from the one in Definition \ref{constr1}.
\end{remark} 

The next proposition shows that $\bar{\bfV}(T)$ and $ \bar{Q}(T)$ from Definition \ref{constr3} are the  Newton iterates for the function $\Phi$. 

\begin{proposition}\label{mainprop}
Let $\bfF$, u, $q(T)$, $\vv(T)$ and $\Phi$ be such that  Assumption \ref{assu3} holds. Then 
$\Lambda(T)$ defined in (\ref{Lambda}) is invertible modulo $q(T)$,   and thus $\bar{\bfV}(T)$ and $ \bar{Q}(T)$ are well  defined in Definition \ref{constr3}. Furthermore 
{\small \begin{equation}
\left[\begin{array}{c}\bar{\bfV}(T)\\
 \bar{Q}(T)-T^d\end{array}\right] = \left[\begin{array}{c}\vv(T)\\q(T)-T^d\end{array}\right]- J_\Phi^{-1}\cdot \left[\begin{array}{c}\bfF(\vv(T))\\\sum_{i=1}^n \lambda_i v_i(T)-T\end{array}\right] \mod q(T),
\end{equation}}
where the vector on the right hand side is $\Phi(\vv(T), q(T)-T^d)$. Finally, we also have that 
$
\sum_{i=1}^n\lambda_i\bar{V}_i(T) =T.$
\end{proposition}

\begin{proof}
Taking derivatives by $T$ of both sides of the equations 
$$F_i(\vv(T))-m_i(T)q(T) = r_i(T)$$
 for $i=1, \dots n$, we get that
\begin{equation*}
J_\bfF(\vv(T)) \frac{\partial \vv(T)}{\partial T} -  \frac{\partial q(T)}{\partial T}  \bfm(T) \equiv   \frac{\partial \bfr(T)}{\partial T} \mod q(T),
\end{equation*}
or equivalently
$$\begin{array}{|ccc|c|c}
%\multicolumn{3}{c}{\scriptsize{n}}&\multicolumn{1}{c}{\scriptsize{1}}\\
\cline{1-4}
      & & & -m_1(T)&\\
      & J_\bfF(\vv(T)) & &\vdots&\\
      & & & -m_n(T)&\\
      \cline{1-4}
  % \lambda_1& \cdots &\lambda_n & 0 &   \\
\cline{1-4} \multicolumn{2}{c}{}
\end{array}\cdot \left[\begin{array}{c}\frac{\partial v_1(T)}{\partial T}\\\vdots \\\frac{\partial v_n(T)}{\partial T} \\\frac{\partial q(T)}{\partial T}\end{array} \right]\equiv 
\left[\begin{array}{c}\frac{\partial r_1(T)}{\partial T}\\\vdots \\\frac{\partial r_n(T)}{\partial T}\end{array}\right]\mod q(T).
$$
From the definition 
$$\bfU(T):=  \frac{\partial \vv(T)}{\partial T} -  \left(J_\bfF (\vv(T))^{-1}\frac{\partial \bfr(T)}{\partial T} \mod q(T)\right)
$$ 
in (\ref{U}) we get that
\begin{equation}\label{kernel}
\begin{array}{|ccc|c|c}
%\multicolumn{3}{c}{\scriptsize{n}}&\multicolumn{1}{c}{\scriptsize{1}}\\
\cline{1-4}
      & & & -m_1(T)&\\
      & J_\bfF(\vv(T)) & &\vdots&\\
      & & & -m_n(T)&\\
      \cline{1-4}
  % \lambda_1& \cdots &\lambda_n & 0 &   \\
\cline{1-4} \multicolumn{2}{c}{}
\end{array}\cdot \left[\begin{array}{c}U_1(T)\\\vdots \\U_n(T)\\\frac{\partial q(T)}{\partial T}\end{array} \right]=
\left[\begin{array}{c}0\\\vdots \\0\end{array}\right]\mod q(T).
\end{equation}
From this and from the definition 
$$
 \Lambda(T)= \sum_{i=1}^n \lambda_i U_i(T) 
 $$
 it is easy to see that 
 $$
 \Lambda(T) = \frac{\partial q(T)}{\partial T} \cdot \left[\lambda_1 \cdots \lambda_n\right] J_\bfF^{-1}(\vv(T)) \left[\begin{array}{c}m_1(T)\\\vdots \\m_n(T)\end{array}\right].
 $$ 
Then using Schur complements we get that 
 $$
\Lambda(T) \cdot  \det J_\bfF(\vv(T)) = \frac{\partial q(T)}{\partial T}\cdot  \det J_\Phi,
 $$
 and since both $\frac{\partial q(T)}{\partial T} $ and $\det (J_\Phi)$ are invertible modulo $q(T)$ by our assumptions, so is $\Lambda(T)$ as was claimed.\\ 
To prove the second claim, consider  $\ww(T)=\vv(T)- J_\bfF(\vv(T))^{-1}\bfF(\vv(T))$ defined in (\ref{ww3}), and define 
\begin{equation}\label{WW}
\bfW(T) := J_\bfF(\vv(T))^{-1}\bfF(\vv(T)).
\end{equation} 
Using that $\sum_{i=1}^n \lambda_i v_i(T)=T $ we get that 
$$
\Delta(T) =\sum_{i=1}^n \lambda_i w_i(T) -T= -\sum_{i=1}^n \lambda_i W_i(T) 
$$ 
and thus  
\begin{equation}\label{uvalue}
\sum_{i=1}^n \lambda_i W_i(T) + \frac{\Delta(T)}{\Lambda(T)} \cdot \sum_{i=1}^n \lambda_i U_i(T)\equiv 0\equiv \sum_{i=1}^n \lambda_i v_i(T)-T \mod q(T).
\end{equation}
Thus from (\ref{kernel}), (\ref{WW}) and (\ref{uvalue}) we have that modulo $q(T)$
{\small 
$$\begin{array}{|ccc|c|c}
%\multicolumn{3}{c}{\scriptsize{n}}&\multicolumn{1}{c}{\scriptsize{1}}\\
\cline{1-4}
      & & & -m_1(T)&\\
      & J_\bfF(\vv(T)) & &\vdots& \\
      & & & -m_n(T)&\\
      \cline{1-4}
   \lambda_1& \cdots &\lambda_n & 0 &   \\
\cline{1-4} \multicolumn{2}{c}{}
\end{array}
\cdot \left[\begin{array}{c} W_1(T) + \frac{\Delta(T)}{\Lambda(T)}U_1(T)\\\vdots \\W_n(T) + \frac{\Delta(T)}{\Lambda(T)}U_n(T)\\\frac{\Delta(T)}{\Lambda(T)}\cdot \frac{\partial q(T)}{\partial T}\end{array} \right]=
\left[\begin{array}{c}F_1(\vv(T))\\\vdots \\F_n(\vv(T)) \\ \sum_{i=1}^n \lambda_i v_i(T)-T  \end{array}\right],
$$
}
or equivalently 
$$
\left[\begin{array}{c} W_1(T) + \frac{\Delta(T)}{\Lambda(T)}U_1(T)\\\vdots \\W_n(T) + \frac{\Delta(T)}{\Lambda(T)}U_n(T)\\\frac{\Delta(T)}{\Lambda(T)}\cdot \frac{\partial q(T)}{\partial T}\end{array} \right] = J_\Phi^{-1} \cdot \left[\begin{array}{c}F_1(\vv(T))\\\vdots \\F_n(\vv(T)) \\ \sum_{i=1}^n \lambda_i v_i(T)-T  \end{array}\right].
$$
This, together with the definitions of $\bar{\bfV}(T)$ in (\ref{bVV}) and  $\bar{Q}(T)$ in (\ref{bQ}) 
 proves the second claim. The third claim is immediate  from  (\ref{uvalue}), as  
 $$
 \sum_{i=1}^n \lambda_i \bar{V}_i=  \sum_{i=1}^n \lambda_i v_i(T) -  \sum_{i=1}^n \lambda_i \left( W_i(T) + \frac{\Delta(T)}{\Lambda(T)}U_i(T)\right) = T-0.
 $$
 \end{proof}

\begin{corollary} Let $R=\K\subset\C$  be a field  equipped with the usual absolute value, and consider the Euclidean norm on the coefficient vectors of polynomials over~$\C$. 
Then the  iteration defined by Definition \ref{constr3} is locally quadratically convergent to an exact RUR of a component of $\langle \bfF \rangle $ over an algebraic extension of $\K$, as long as  Assumption \ref{assu3} is satisfied in each iteration.
\end{corollary}

%------------------------------------------------------------------------------

\section{Example: A cubic-centered 12-bar linkage}

%------------------------------------------------------------------------------

To illustrate the application of these techniques, we compute 
an RUR for a rational component of a square system to prove that 
it solves an overdetermined system of equations arising from 
a 12-bar spherical linkage.  The overdetermined polynomial system $\bfG$
consists of $17$ quadratic and $2$ linear polynomials in $18$ variables 
for the linkage first described in \cite{WLE07}, which 
is presented in \cite[Fig.~3]{WHS11}.
The trivial rotation of the cube is removed by placing the center of the cube
at the origin and fixing two adjacent vertices, say $P_7$ and $P_8$,
at $(-1,1,-1)$ and $(-1,-1,-1)$, respectively.  The $18$ variables
are the coordinates of the six remaining vertices of the cube, say $P_1,\dots,P_6$
with $P_i = (P_{ix},P_{iy},P_{iz})$. 
The $17$ quadratic conditions force these free vertices to maintain
their relative distances:
$$\begin{array}{ll}
\|P_i - P_j\|^2 - 4 = 0,
&(i,j)\in\left\{\begin{array}{c} (1,2),(3,4),(5,6),(1,5),(2,6),(3,7),\\
(4,8),(1,3),(2,4),(5,7),(6,8)\end{array}\right\} \\ \\
\|P_i\|^2 - 3 = 0, &i = 1,\dots,6.
\end{array}$$

The irreducible components of these $17$ quadratic polynomials was 
first described in \cite[Table~1]{H13}.  This 
decomposition shows that there is a unique irreducible surface $S$
of degree $16$, which is the current focus of study.  
In particular, the rational component is the $16$ points
arising from the intersection of $S$ with the codimension two
linear space defined by:
$$\begin{array}{l}
P_{3x} + P_{4x} + P_{2z} = 0, \\
P_{5x} - P_{6x} + P_{3y} + P_{3z} + 1 = 0.\\
\end{array}$$

In order to compute an RUR for this rational component, we consider
the square polynomial system $\bfF$ consisting of these two linear equations
and $16$ quadratic equations obtained by adding the last
quadratic above, i.e., \mbox{$\|P_6\|^2 - 3$}, to the other sixteen.
Starting with a witness set for $S$, we used {\tt Bertini} \cite{BHSW06} 
to compute numerical approximations of the $16$ points of interest.  
From these points, we observe that the 
variable $P_{6z}$ is distinct, so we take the primitive element
$u(P_1,\dots,P_6) = P_{6z} = T.$
Next, we produce an initial guess for the monic univariate 
polynomial $q(T)$ of degree $16$.  Since $q(T)$ naturally has
small integer coefficients, this polynomial was computed exactly.
We produced an initial guess for the univariate 
polynomials ${\bf v}(T)$ of degree at most $15$ via Lagrange interpolation
using the computed numerical approximations. These are polynomials with rational coefficients having at most 5 digit
numerators and denominators. 
 These initial guesses are:

{\tiny $$
\begin{array}{lcl}
q(T) &=& T^{16} + 20 T^{15} + 210 T^{14} + 1230 T^{13} + 4212 T^{12} + 4677 T^{11} \\
& & -~6886 T^{10} - 21389 T^9 + 58242 T^8 - 45269 T^7 - 6118 T^6 \\
& & +~58968 T^5 - 103014 T^4 + 119847 T^3 - 91281 T^2 + 40466 T - 8291 \\
v_{1x}(T) &=& -1 \\
v_{2x}(T) &=& -1 \\
v_{3x}(T) &=& -1/1112 T^{15} - 1/57 T^{14} - 13/72 T^{13} - 139/138 T^{12} - 167/54 T^{11} \\
& &  -~67/60 T^{10} + 1616/117 T^{9} + 3173/120 T^{8} - 3922/63 T^{7} \\
& & +~1165/39 T^{6} + 2909/123 T^{5} - 1709/30 T^{4} + 9965/109 T^{3} \\
& & -~2366/23 T^{2} + 63032/1027 T - 1466/99 \\
v_{4x}(T) &=& 1/1112 T^{15} + 1/57 T^{14} + 13/72 T^{13} + 139/138 T^{12} + 167/54 T^{11}\\
& & +~67/60 T^{10} - 1616/117 T^{9} - 3173/120 T^{8} + 3922/63 T^{7} \\
& & -~1165/39 T^{6} - 2909/123 T^{5} + 1709/30 T^{4} - 9965/109 T^{3} \\
& & +~2366/23 T^{2} - 63032/1027 T + 1565/99  \\
v_{5x}(T) &=& -1/10090 T^{15} -1/506 T^{14} - 1/49 T^{13} - 3/26 T^{12} - 33/94 T^{11} \\
& & -~7/132 T^{10} + 187/83 T^{9} + 538/107 T^{8} - 398/77 T^{7} \\
& & -~1601/587 T^{6} + 2223/502 T^{5} - 129/113 T^{4} + 544/63 T^{3} \\
& & -~917/150 T^{2} - 233/97 T + 23/36  \\
v_{6x}(T) &=& -1/5519 T^{15} - 1/265 T^{14} - 3/73 T^{13} - 31/121 T^{12} - 33/34 T^{11} \\
& & -~129/79 T^{10} - 37/517 T^{9} + 233/61 T^{8} - 2017/270 T^{7} \\
& & +~373/177 T^{6} + 155/53 T^{5} - 2314/271 T^{4} + 705/59 T^{3} \\
& & -~4487/358 T^{2} + 963/158 T - 59/47 \\
v_{1y}(T) &=& 1 \\
v_{2y}(T) &=& -1 \\
v_{3y}(T) &=& -1/2038 T^{15} - 1/103 T^{14} - 17/169 T^{13} - 31/54 T^{12} - 245/132 T^{11} \\
& & -~120/89 T^{10} + 224/39 T^{9} + 757/60 T^{8} - 581/18 T^{7} \\
& & +~3193/184 T^{6} + 1661/150 T^{5} - 21047/654 T^{4} + 7721/163 T^{3} \\
& & -~7541/138 T^{2} + 2725/78 T - 785/94  \\
v_{4y}(T) &=& -1/2357 T^{15} - 1/108 T^{14} - 33/314 T^{13} - 39/56 T^{12} - 187/65 T^{11} \\
& & -~581/95 T^{10} - 2401/600 T^{9} + 930/101 T^{8} - 775/113 T^{7} \\
& & -~445/44 T^{6} + 1237/142 T^{5} - 1487/168 T^{4} + 1817/150 T^{3} \\
& & -~462/167 T^{2} - 1489/248 T + 494/135 \\
v_{5y}(T) &=& 1/12979 T^{15} + 1/694 T^{14} + 1/70 T^{13} + 11/148 T^{12} + 37/177 T^{11} \\
& & +~1/107 T^{10} - 213/365 T^{9} + 118/195 T^{8} + 1551/163 T^{7} \\
& & -~580/67 T^{6} - 531/166 T^{5} + 1893/211 T^{4} - 463/51 T^{3} \\
& & +~1090/61 T^{2} - 2787/212 T + 258/95 \\
v_{6y}(T) &=& 1/5519 T^{15} + 1/265 T^{14} + 3/73 T^{13} + 31/121 T^{12} + 33/34 T^{11} \\
& & +~129/79 T^{10} + 37/517 T^{9} - 233/61 T^{8} + 2017/270 T^{7} \\
& & -~373/177 T^{6} - 155/53 T^{5} + 2314/271 T^{4} - 705/59 T^{3} \\
& & +~4487/358 T^{2} - 1121/158 T + 12/47\\
v_{1z}(T) &=& -1 \\
v_{2z}(T) &=& -1 \\
v_{3z}(T) &=& 1/2448 T^{15} + 1/126 T^{14} + 2/25 T^{13} + 13/30 T^{12} + 183/148 T^{11} \\
& & -~22/95 T^{10} - 1299/161 T^{9} - 553/40 T^{8} + 1259/42 T^{7} \\
& & -~338/27 T^{6} - 3182/253 T^{5} + 1958/79 T^{4} - 1630/37 T^{3} \\
& & +~4292/89 T^{2} - 1084/41 T + 1206/221  \\
v_{4z}(T) &=& -1/2106 T^{15} - 1/119 T^{14} - 4/53 T^{13} - 23/74 T^{12} - 25/116 T^{11} \\
& & +~5654/1131 T^{10} + 2868/161 T^{9} + 1327/77 T^{8} - 7423/134 T^{7} \\
& & +~2759/69 T^{6} + 1718/115 T^{5} - 9575/199 T^{4} + 5393/68 T^{3} \\
& & -~12613/126 T^{2} + 6401/95 T - 1883/92 \\
v_{5z}(T) &=& 1/5677 T^{15} + 1/292 T^{14} + 7/202 T^{13} + 26/137 T^{12} + 219/391 T^{11} \\
& & +~22/353 T^{10} - 139/49 T^{9} - 774/175 T^{8} + 279/19 T^{7} \\
& & -~587/99 T^{6} - 450/59 T^{5} + 536/53 T^{4} - 3029/171 T^{3} \\
& & +~1343/56 T^{2} - 462/43 T + 14/13  \\
v_{6z}(T) &=& T
\end{array}
$$
}

We refined the approximate RUR using Algorithm \ref{algroots}, and in 3 iterations (in roughly 1 second) we found the exact RUR.  In the  ${\bf v}(T)$ polynomials of the exact RUR, 
the numerators and denominators of the coefficients have at most 28 digits, namely:

{\tiny
$$
\begin{array}{lcl}
\alpha &=& 1/3204471773221369279790658525 \\
v_{3x}(T) &=& \alpha (-2881129493593630865610329 T^{15}-56469358709164889119641644T^{14}\\
& & -~578442048083015317390422659 T^{13}-3227775460749576025678391459 T^{12}\\
& & -~9909894946587188228883719582 T^{11}-3578358749346900975113448620 T^{10}\\
& & +~44260151084205755500190960589 T^{9}+84731577601881128711018565420 T^{8}\\
& & -~199491165378780802464515305188 T^{7}+95723229838339681423971314578 T^{6}\\
& & +~75787130941751596487093105995 T^{5}-182548470032615020420523374937 T^{4}\\
& & +~292959497003500175534452099849 T^{3}-329643042476857281605069314889 T^{2}\\
& & +~196674125364601362085119810025 T-47452126308845628915580789974) \\
v_{4x}(T) &=& \alpha(2881129493593630865610329 T^{15}+56469358709164889119641644 T^{14}\\
& & +~578442048083015317390422659 T^{13}+3227775460749576025678391459 T^{12}\\
& & +~9909894946587188228883719582 T^{11}+3578358749346900975113448620 T^{10}\\
& & -~44260151084205755500190960589 T^{9}-84731577601881128711018565420 T^{8}\\
& & +~199491165378780802464515305188 T^{7}-95723229838339681423971314578 T^{6}\\
& & -~75787130941751596487093105995 T^{5}+182548470032615020420523374937 T^{4}\\
& & -~292959497003500175534452099849 T^{3}+329643042476857281605069314889 T^{2}\\
& & -~196674125364601362085119810025 T+50656598082066998195371448499) \\
v_{5x}(T) &=& \alpha/47(-14926508985814692725660295 T^{15}-297773391799482191643772465 T^{14}\\
& & -~3081610308415501202687193085 T^{13}-17389506895894299310438310140 T^{12}\\
& & -~52870530884177916409158871660 T^{11}-7985554397153232406118329735 T^{10}\\
& & +~339337071563230521110991064610 T^{9}+757269429739682916272675120805 T^{8}\\
& & -~778491818056744498922090610960 T^{7}-410778307074659870607871725630 T^{6}\\
& & +~666945252290852928200793739220 T^{5}-171936037274621562789031991275 T^{4}\\
& & +~1300517229936781609262733002915 T^{3}-920730755436417716152608059995 T^{2}\\
& & -~361767415638579519536847568045 T+96208631064260278520898204560)\\
v_{6x}(T) &=& \alpha/47(-27287280868886188971151890 T^{15}-567883140744448436140744235 T^{14}\\
& &-~6190772706188356762139486930 T^{13}-38580721830051962511624944735 T^{12}\\
& &-~146182732583672120750472401420 T^{11}-245939987210083103698877745635 T^{10}\\
& &-~10778669490043487418881186945 T^{9}+575272052935709136436375750005 T^{8}\\
& &-~1125114039605827435077357466500 T^{7}+317389390835062387309791675960 T^{6}\\
& &+~440466388670191221545685003445 T^{5}-1286022849359951198330786683160 T^{4}\\
& &+~1799659504525177884513418973560 T^{3}-1887676569404811089807938408490 T^{2}\\
& &+~917961856537764412658257692880 T-189055429157249714249696125460)\\
%\end{array}$$

%$$
%\begin{array}{lcl}
v_{3y}(T) &=& \alpha(-1572062320020980286480607 T^{15}-31108187322082085458724777 T^{14}\\
& &-~322297219762495483795767647 T^{13}-1839326187121146132000755672 T^{12}\\
& &-~5947630470096723947860088831 T^{11}-4320609510981215075777569160 T^{10}\\
& &+~18405439998982941489139392512 T^{9}+40429646494515311378527374510 T^{8}\\
& &-~103433053131401921829654108504 T^{7}+55608079790549864732386331774 T^{6}\\
& &+~35484215857890035406790077085 T^{5}-103126222378917399950067588446 T^{4}\\
& &+~151789772699498984312446113442 T^{3}-175108178833837080947804129237 T^{2}\\
& &+~111951203875666042023358897150 T-26760787199332495231945653817)\\
v_{4y}(T) &=& \alpha(-1359749822277659136367732 T^{15}-29651222724814187524939277 T^{14}\\
& &-~336783507299580007515911147 T^{13}-2231865653484431345450780072 T^{12}\\
& &-~9219287203307410697519852906 T^{11}-19597877374389843525839721035 T^{10}\\
& &-~12823225261355870172371961013 T^{9}+29506643456688254696406280185 T^{8}\\
& &-~21977689644205909107128376729 T^{7}-32409171956862700073069418826 T^{6}\\
& &+~27915029062394568996017596135 T^{5}-28363400298198409791080814221 T^{4}\\
& &+~38816863323840423843221276617 T^{3}-8865121893390327106993539437 T^{2}\\
& &-~19239716389698007639434041750 T+11725943590084476919258204283) \\
v_{5y}(T) &=& \alpha/47(11603983966259833376058160 T^{15}+217153170339381974795502905 T^{14}\\
& &+~2139063241115082817328441135 T^{13}+11191289970084093849864343715 T^{12}\\
& &+~31486473569992118869518659780 T^{11}+1401302126019324290665142660 T^{10}\\
& &-~87890587988818445033654819495 T^{9}+91139458508898945030483080370 T^{8}\\
& &+~1433108663902588131761054475525 T^{7}-1303792568225360896022867456000 T^{6}\\
& &-~481771480980237531202595548510 T^{5}+1351208978813262782720191318100 T^{4}\\
& &-~1367308841574912038190287285515 T^{3}+2691219896341123859883175304435 T^{2}\\
& &-~1979954015860581566552572962955 T+409019746222899610159043835865) \\
v_{6y}(T) &=& \alpha/47(27287280868886188971151890 T^{15}+567883140744448436140744235 T^{14}\\
& &+~6190772706188356762139486930 T^{13}+38580721830051962511624944735 T^{12}\\
& &+~146182732583672120750472401420 T^{11}+245939987210083103698877745635 T^{10}\\
& &+~10778669490043487418881186945 T^9-575272052935709136436375750005 T^8\\
& &+~1125114039605827435077357466500 T^7-317389390835062387309791675960 T^6\\
& &-~440466388670191221545685003445 T^5+1286022849359951198330786683160 T^4\\
& &-~1799659504525177884513418973560 T^3+1887676569404811089807938408490 T^2\\
& &-~1068572029879168768808418643555 T+38445255815845358099535174785)\\
v_{3z}(T) &=& \alpha(1309067173572650579129722 T^{15}+25361171387082803660916867 T^{14}\\
& &+~256144828320519833594655012 T^{13}+1388449273628429893677635787 T^{12}\\
& &+~3962264476490464281023630751 T^{11}-742250761634314100664120540 T^{10}\\
& &-~25854711085222814011051568077 T^{9}-44301931107365817332491190910 T^{8}\\
& &+~96058112247378880634861196684 T^{7}-40115150047789816691584982804 T^{6}\\
& &-~40302915083861561080303028910 T^{5}+79422247653697620470455786491 T^{4}\\
& &-~141169724304001191222005986407 T^{3}+154534863643020200657265185652 T^{2}\\
& &-~84722921488935320061760912875 T+17486867336291764403844477632) \\
%\end{array}$$

%$$
%\begin{array}{lcl}
v_{4z}(T) &=& \alpha(-1521379671315971729242597 T^{15}-26818135984350701594702367 T^{14}\\
& &-~241658540783435309874511512 T^{13}-995909807265144680227611387 T^{12}\\
& &-~690607743279777531363866676 T^{11}+16019518625042942550726272415 T^{10}\\
& &+~57083376345561625672562921602 T^{9}+55224934145192874014612285235 T^{8}\\
& &-~177513475734574893357386928459 T^{7}+128132401795202381497040733404 T^{6}\\
& &+~47872101879357027491075509860 T^{5}-154185069734416610629442560716 T^{4}\\
& &+~254142633679659751691230823232 T^{3}-320777920583466954498075775452 T^{2}\\
& &+~215913841754299369724553851775 T-65587013445372844394420311307) \\
v_{5z}(T) &=& \alpha/47(26530492952074526101718455 T^{15}+514926562138864166439275370 T^{14}\\
& &+~5220673549530584020015634220 T^{13}+28580796865978393160302653855 T^{12}\\
& &+~84357004454170035278677531440 T^{11}+9386856523172556696783472395 T^{10}\\
& &-~427227659552048966144645884105 T^{9}-666129971230783971242192040435 T^{8}\\
& &+~2211600481959332630683145086485 T^{7}-893014261150701025414995730370 T^{6}\\
& &-~1148716733271090459403389287730 T^{5}+1523145016087884345509223309375 T^{4}\\
& &-~2667826071511693647453020288430 T^{3}+3611950651777541576035783364430 T^{2}\\
& &-~1618186600222002047015725394910 T+162200941817234975487984680630) \\
\end{array}$$}
We checked $\bfF({\bf v}(T)) \equiv 0 \hbox{~~mod~} q(T)$ 
and $\bfG({\bf v}(T))\equiv 0 \hbox{~~mod~} q(T)$ using exact arithmetics,
meaning that we found an exact rational component of our zero dimensional ideal
and proves it satisfies the overdetermined system of equations.

\section{\bf Parallel complexity}

In this section, we study two parallel algorithms for our two constructions in Definitions \ref{constr2} and \ref{constr3}, and analyze their parallel complexity. We express our  complexity results as functions of  the number of variables $n$ and the number of roots $d$. In many applications, the number of roots $d$ is large, possibly being an exponential function of $n$. Our goal is to demonstrate that  we can efficiently distribute our computations to polynomially many processors in $n$ and $d$ so that the parallel computational time is polynomial in  $\log(d)$ and $n$. 

Note that for large $d$, the bottleneck of the computation of the formulas of Definitions \ref{constr2} and \ref{constr3} is the parallel computation  of modular inverses of polynomials   modulo the   degree $d$ polynomial $q(T)$. We consider  two approaches to compute modular inverses, and to modular arithmetic in general:
\begin{enumerate}
\item Via the computation of the roots of the modulus $q(T)$ and using parallel interpolation algorithms at these roots.
\item Via  the parallel solution of Toeplitz-like linear systems.
\end{enumerate} 

Our first construction in Definition \ref{constr2} is particularly well suited for the approach via the roots of the modulus $q(T)$ due to Proposition  \ref{roots}. In fact, the roots of $q(T)$ are combinations of the coordinates of approximate roots of a component of the input system $\bfF=(F_1, \ldots, F_n)$, and according to Proposition  \ref{roots}, the polynomials $\tilde{Q}(T), \tilde{\bfV}(T)=(\tilde{V}_1(T), \ldots, \tilde{V}_n(T))$  in Definitions \ref{constr2} can be obtained from the coordinates of higher accuracy roots by using the interpolation formulas in (\ref{qT}) and (\ref{vj}).  
Thus, we propose to compute one iteration for our first construction by using the underlying local data and interpolation.  Below, in Subsection \ref{sec:compl1}, we  analyze the parallel complexity of this algorithm. 

For our second construction in Definition \ref{constr3}, unfortunately  we do not have any results to connect the roots of the modulus with the common roots of the input $\bfF=(F_1, \ldots, F_n)$. In theory, we could approximate the roots of our modulus $q(T)$ in each iteration. However, to obtain exact modular arithmetics modulo $q(T)$ using these approximate roots would require further considerations, which we bypass in this paper, offering instead the algorithm for our first construction outlined above. 

Instead, for our second construction we propose the parallel solution of a  general Toeplitz-like linear system of equations,
which refines and modifies the algorithm of \cite{Pan1992} by using a more efficient displacement representation 
with factor circulant matrices, defined in \cite[Example 4.4.2]{Pan2001}, rather than triangular Toeplitz matrices.  This improvement will reduce the number of required FFT computations by a factor of $2$.  Below, in Subsection \ref{compl2}, we detail this new modified algorithm and analyze its parallel  complexity. Then, we apply these complexity bounds for the computation of modular inverses and for the computation of the polynomials in Definition \ref{constr3}. Finally, we briefly discuss 
the computation of modular inverses  (or rather cofactors)
when all the roots of at least one of the two
associated input polynomials are simple and known.

\subsection{Computational model}

We have weighed the possibility of using  several parallel computational models  in our complexity estimates, including message passing interface (MPI) that measures communication costs related to passing data, and PRAM (Parallel Random Access Machine) models that assume that data is readily available to all processors via shared memory. 

From a computational standpoint, the current ``gold standard" is a hybrid model combining both message passing and shared memory.  For example, a realistic cluster (the one that we used) consists of 13 nodes with each node having 64 processors that share memory -- one would like parallel programming to send data between the nodes that is then shared amongst all the local processors.

Beyond choosing between different architectures, we also have to decide between an algebraic or a Boolean  computational model.  The algebraic model assumes that the  basic arithmetic operations in the coefficient field can be done at unit cost, independently of the size of the numbers appearing in the computation, while under the Boolean model the estimates depend on the size of the numbers as well. 

In our applications, the iterations under study  will be conducted using floating point complex numbers as coefficients.
However, for the purposes of the complexity analysis, in each iterations we assume exact rational arithmetic using the input floating point numbers as exact rational numbers, and at the end of the iteration we round them to the desired precision. Alternatively, we can round the numbers after each arithmetic operations to the desired precision.
Thus, up to a constant multiple that depends on the desired precision,  we can assume that at each iteration we are dealing with rational numbers or Gaussian rationals, and arithmetics on them has unit cost (cf. \cite{BlCuShSm}). That is why
we choose the algebraic computational model and not the Boolean one  to estimate the parallel complexity of our iterations. 

Ultimately we choose the PRAM  arithmetic model \cite{KarpRam1990}, in which we more conveniently expose our complexity estimates, but we can readily obtain from our algorithms
similar estimates in terms of basic operations (like FFT),
which are efficient under any  reasonable model. We will invoke Brent's scheduling principle 
that allows us to save processors by slowing down the computations, so that $O_A(t,p)$ will denote
the simultaneous upper bounds $O(ts)$ on the parallel arithmetic time, and $\lceil p/s\rceil$  on the number
of processors involved, where any $s\geq 1$ can be assumed.

\subsection{Parallel complexity of the first  construction from the roots}\label{sec:compl1}

By Proposition  \ref{roots}, the iterates of Definition \ref{constr2} are the same as the Lagrange interpolants of the approximate roots obtained from one step  local Newton iteration. We assume here that the coordinates of the approximate roots are given as floating point complex numbers,  thus our base field is  $\K:=\Q(i)$. 
Below we give estimates on the parallel complexity of the following simple algorithm:

%------------------------------------------------------------------------------

\begin{algorithm}\label{algroots} {\rm Computation of RUR from roots.}

%------------------------------------------------------------------------------

\begin{description}

%------------------------------------------------------------------------------

\item[{\sc Input:}]  A primitive element $u=\lambda_1x_1+\ldots +\lambda_nx_n$ and approximate roots $\bfz_1,\dots,\bfz_d\in \K^n$. We assume that the corresponding RUR $u, q(T), {\bf v}(T)$ satisfies Assumption \ref{qassu} (but  need not to be given explicitly).

%------------------------------------------------------------------------------

\item[{\sc Output:}] The updated RUR $\tilde{Q}(T), \tilde{\bfV}(T)$ defined in Definition \ref{constr2}, and its common roots  $\tilde{\bfz}_1,\dots,\tilde{\bfz}_d\in \K^n$.
%------------------------------------------------------------------------------

%\item[{\sc Initialization:}] Set $X_0\leftarrow I$, $A\leftarrow I-\lambda B$  
% for a scalar parameter $\lambda $, $d\leftarrow \lceil \log_2(k+1)\rceil$. 

%------------------------------------------------------------------------------

\item[{\sc Computations:}] $~$

%------------------------------------------------------------------------------
%------------------------------------------------------------------------------

\begin{enumerate}
  \item %1
  Compute $
\tilde{\bfz}_i:= \bfz_i- J_\bfF (\bfz_i)^{-1}\bfF(\bfz_i) \quad i=1, \ldots, d 
$
\item %2
Compute $\tilde{Q}(T):=\prod_{i=1}^d (T-u(\tilde{\bfz}_i))$
\item %3
Interpolate the polynomials $\tilde{V}_1(T), \ldots, \tilde{V}_n(T)$ such that  $\tilde{V}_j(u(\tilde{\bfz}_i))=\tilde{z}_{i,j}$ for $i=1, \ldots, d$ and $j=1, \ldots n$.
\end{enumerate}

%------------------------------------------------------------------------------

\end{description}

%------------------------------------------------------------------------------

\end{algorithm}

The next proposition gives the  complexity bounds for  Algorithm  \ref{algroots}. In what follows we use the following notation:
\begin{itemize}
\item $\log^{(0)}(N):=N$, $\log^{(h)}(N):=\log_2(\log^{(h-1)}(N))$,
for $h\geq 1$, and $\log^*(N):=\max \{h:~\log ^{(h)}(N)>0\}$; 
\item We use the exponent $\omega$ to denote a number such that $O(n^\omega)$ arithmetic operations are
sufficient for the multiplication of two  $n\times n$ matrices. Note that $ 2\le \omega\le 2.373$.
\end{itemize}

\begin{proposition}\label{compl1}
Given  $u=\sum_{i=1}^n\lambda_i x_i$, and  $\bfz_1, \ldots, \bfz_d\in \K^n$ satisfying the assumptions of Algorithm   \ref{algroots}.   Then, we can compute the polynomials   $\tilde{Q}(T), \tilde{\bfV}(T)$  of Definition \ref{constr2} and the corresponding approximate roots  $\tilde{\bfz}_1, \ldots, \tilde{\bfz}_d\in \K^n$ of $\bfF$ 
%!!
in two stages with respective costs 
$$O_A(\log^2(n),n^{\omega+1})\text{ and } O_A(\log^2(d) \log^*(nd), nd/\log^*(d)).$$
%!! parallel complexity  
%!! $O_A(n^\omega+ log^2(d),nd\log^*(d))$. 
\end{proposition}
 
 \begin{proof}
The $O_A(\log^2(n),n^{\omega+1})$ term comes from Step 1 using  \cite[page 319]{BiniPan1994}.  In Step 2 and 3 we apply the multipoint polynomial evaluation and polynomial interpolation algorithms in 
Sec 3.1, 3.3, and 3.13 of \cite{Pan2001}. These algorithms  
essentially amount to 
$O(\log (d))$ steps each performing concurrent  univariate  multiplications and divisions
of polynomials having degrees at most $d$.
We can perform
these operations in time $O(\log^2(d))$ using 
$O(nd/\log^*(d))$ processors.   \end{proof}

\subsection{Parallel complexity of the second construction using modular arithmetics}\label{compl2}

In this subsection, we analyze the parallel complexity of one iteration defined in the polynomial arithmetics modulo $q(T)$  in  Definition \ref{constr3}.

In Definition \ref{constr3}, we have to compute two  modular inverses:
$$\frac{1}{\det J_\bfF(\bfv (T))}\text{ and } \frac{1}{\Lambda(T)} \mod q(T),
$$ 
which are the bottleneck of the computation if $d=\deg q(T)$ is large in comparison to $n$. 
One way to compute modular inverses  is via  the solution of a non-singular linear system with  Sylvester coefficient matrix, which is a Toeplitz-like structured matrix.

%------------------------------------------------------------------------------

In what follows, we describe an algorithm  for the  solution of a  general Toeplitz-like linear system of equations,
which refines the algorithm of \cite{Pan1992} with  more efficient displacement representation of \cite[Example 4.4.2]{Pan2001}
using  factor circulant matrices. 
Then, we briefly cover the special case where all roots of the input polynomials 
or at least one of them are  simple and available, in which case 
even more efficient algorithms can be applied.

%------------------------------------------------------------------------------
  
\subsubsection{Parametrized Newton's iteration.}
  	
%------------------------------------------------------------------------------

Hereafter $I_m$ denotes the $m\times m$ identity matrix.
We begin with recalling parametrized Newton's iteration that computes a sequence of the powers of a matrix.

%------------------------------------------------------------------------------

\begin{algorithm}\label{algnwt} {\em Parametrized Newton's iteration}. 

%------------------------------------------------------------------------------

\begin{description}

%------------------------------------------------------------------------------

\item[{\sc Input:}] two positive integers $k$ and $m$ and
an $m\times m$  matrix $B$.

%------------------------------------------------------------------------------

\item[{\sc Output:}] the powers $I_m=B^0$, $B$, $B^2,\dots$, $B^k$,
defined as the coefficients of the matrix polynomial $A^{-1}\mod \lambda^{k+1}$ 
for $A=I_m-\lambda B$. 

%------------------------------------------------------------------------------

\item[{\sc Initialization:}] Set $X_0\leftarrow I_m$, $A\leftarrow I_m-\lambda B$  
 for a scalar parameter $\lambda $, $p\leftarrow \lceil \log_2(k+1)\rceil$. 

%------------------------------------------------------------------------------

\item[{\sc Computations:}] $~$

%------------------------------------------------------------------------------

%\begin{enumerate}

%------------------------------------------------------------------------------

{\bf Stage i}, $i=1,\dots,p$. Compute  
$X_{i}=X_{i-1}(2I_m-AX_{i-1})$, the matrix polynomial in $\lambda $.
 Output the matrix polynomial
$X_p\mod \lambda^{k+1}$.

%------------------------------------------------------------------------------

\end{description}

%------------------------------------------------------------------------------

\end{algorithm}

%------------------------------------------------------------------------------

Paper \cite{Pan1992} first shows that  
$X_i=A^{-1}\mod \lambda^{2^i}=\sum_{j=0}^{2^i-1}(\lambda B)^j$ for all $i$
over any ring of constants,
then proves correctness of the algorithm, and finally extends it to solving a nonsingular linear system 
$B{\bf y}={\bf f}$ as follows.

%------------------------------------------------------------------------------

\begin{algorithm}\label{algnul} {\em Extension to LIN$\cdot$SOLVE}. 

%------------------------------------------------------------------------------

\begin{description}

%------------------------------------------------------------------------------

\item[{\sc Input:}] two positive integers  $m$, a vector ${\bf f}$ of dimension $m$,
and the powers $I_m=B^0$, $B$, $B^2,\dots$, $B^m$ of  a nonsingular $m\times m$  matrix $B$,
defined as the coefficients of the matrix polynomial $A^{-1}\mod \lambda^{m+1}$ 
for $A=I_m-\lambda B$.

%------------------------------------------------------------------------------

\item[{\sc Output:}] the vector  ${\bf y}=B^{-1}{\bf f}$. 

%------------------------------------------------------------------------------

%\item[{\sc Initialization:}] Set $X_0\leftarrow I$, $A\leftarrow I-\lambda B$  
% for a scalar parameter $\lambda $, $d\leftarrow \lceil \log_2(k+1)\rceil$. 

%------------------------------------------------------------------------------

\item[{\sc Computations:}] $~$

%------------------------------------------------------------------------------
%------------------------------------------------------------------------------

\begin{enumerate}

%------------------------------------------------------------------------------

\item %1
Compute the traces of the matrices $B^i$, $i=1,2,\dots,m$ as the coefficients of 
 the trace of the matrix $A^{-1}\mod \lambda^{m+1}$,
which is a matrix polynomial in $\lambda$.

%------------------------------------------------------------------------------

\item %2
Compute the coefficients $c_0,\dots,c_{m-1}$ of the characteristic polynomial 
$c_B=\sum_{j=0}^{m}c_i\lambda ^i=\det(\lambda I_m-B)$                                                               
%matrix $C \leftarrow A+UV^H$.
%If this matrix is rank deficient, then either output FAILURE and stop if $k\ge \nu$ 
%or otherwise set $k\leftarrow k+1$ and go to Stage 1.

%------------------------------------------------------------------------------

\item %3
Note that $c_0\neq 0$ because the matrix $B$ is assumed to be nonsingular, write $c_m=1$,
and compute and output the vector  ${\bf y}=B^{-1}{\bf f}=-\sum_{j=0}^{m}(c_i/c_0)B^{i-1}{\bf f}$.

%------------------------------------------------------------------------------

\end{enumerate}

%------------------------------------------------------------------------------

\end{description}

%------------------------------------------------------------------------------

\end{algorithm}

%------------------------------------------------------------------------------

There are more efficient parallel algorithms for the solution of 
a general linear system of equations, but 
next we refine it to obtain a
superior parallel algorithm  in the case of 
 Toeplitz-like matrices $B$. We first recall some relevant definitions.

%------------------------------------------------------------------------------
  
\subsubsection{\bf  Toeplitz-like matrices: fundamentals.}\label{stlf}
  	
%------------------------------------------------------------------------------

Write $J_1=(1)$ is a $1\times 1$ matrix,
$J_k=\begin{pmatrix} {\bf 0}^T & 1\\ J_{k-1} & {\bf 0}\end{pmatrix}$ for $k=2,3,\dots,m$,
$J=J_m$ is the $m\times m$ reflection matrix. 
$Z_f=\begin{pmatrix} {\bf 0}^T & f\\ I_{m-1} & {\bf 0}\end{pmatrix}$
% denotes
denotes the  $n\times n$ matrix of the $f$-circular shift
for a scalar $f\neq 0$. 
%such that
%\begin{equation}\label{eqef}
%Z_e-Z_{f}=(e-f){\bf e}_1{\bf e}_n^T,
%\end{equation}
%\begin{equation}\label{eqrever}
%$JZ_fJ=Z_{f}^T,~~
%\end{equation}
%\begin{equation}\label{eqrevert}
%JZ_f^TJ=Z_{f}$ 
%\end{equation}
%for any pairs of scalars $e$ and $f$, and if $f\neq 0$, then
%\begin{equation}\label{eqinv}
%$Z_f^{-1}=Z_{1/f}^T$.
%\end{equation}

$Z_f({\bf v})=\sum_{i=0}^{m-1}v_{i}Z_f^{i}$
is an  $f$-circulant matrix,
defined by its first column ${\bf v}=(v_i)_{i=0}^{m-1}$
and a scalar $f\neq 0$ and
called circulant for $f=1$.
These matrices belong to the class of Toeplitz matrices
${\displaystyle\left[t_{i-j}\right]_{i,j=0}^{m-1}}$, which in turn can be extended
to the class $\mathcal T$ of Toeplitz-like matrices.
The 
{\em Sylvester
displacements} $Z_eT-TZ_f$ of a Toeplitz-like matrix $T$
 has small ranks (meant ``small" in context)
for a fixed pair of distinct scalars $e$ and $f$, e.g., $e=1$, $f=-1$.
These ranks are called {\em displacement ranks}.
In particular 
Toeplitz, Sylvester and Frobenius companion 
 matrices of any size
have displacement ranks at most 2. 

Recall that an $m\times m$ matrix $M$ of a rank at most $r$
can be expressed nonuniquely through its {\em generator} $(G,H)$ of length $r$,
\begin{equation}\label{eqdsp}
M=\sum_{i=1}^d{\bf g}_i{\bf h}_i^T=GH^T
\end{equation}
 where $G=({\bf g}_1~\cdots~{\bf g}_r)$
and $H=({\bf h}_1~\cdots~{\bf h}_r)$ 
%\begin{equation}
%G=\begin{pmatrix}g_{11}&\cdots&g_{1l}\\ g_{21}&\cdots&g_{2l}\\ \vdots&&\vdots\\ g_{n1}&\cdots&g_{nl}\end{pmatrix}=({\bf g}_1~\cdots~{\bf g}_l),~~H=\begin{pmatrix}h_{11}&\cdots&h_{1l}\\ h_{21}&\cdots&h_{2l}\\ \vdots&&\vdots\\ h_{n1}&\cdots&h_{nl}\end{pmatrix}=({\bf h}_1~\cdots~{\bf h}_l), 
%\end{equation}
and call a generator for a displacement of a matrix   
a {\em displacement generator} for the matrix itself.
  \begin{theorem}\label{thdsp} (See % \cite[Example 4.4.2]{P01}
\cite[Example 4.4.2]{Pan2001}.)
Assume a displacement generator $(G,H)$ of a length $r$ in (\ref{eqdsp})
for a matrix $Z_eT-TZ_f$ where $e\neq f$. Then this matrix can be expressed as follows:
$(e-f)T=\sum_{i=1}^rdZ_e({\bf g}_i)Z_f(J{\bf h}_i)^T$.
  \end{theorem}

The latter compressed representation of an $m\times m$ matrix $T$
 uses $2rm$ parameters rather than $m^2$ entries and enables fast 
multiplication of the matrix by a vector if $m\gg r$. Indeed,
the theorem reduces this operation to $r$ concurrent multiplications
of  $e$-circulant matrices by vectors, 
followed by $r$ such  multiplications by $f$-circulant matrices,
 and finally to the summation of $r$ vectors.
Next, we estimate the computational cost of such a multiplication. 

 \begin{theorem}\label{thtmpm} (See % \cite[equations (2.4.3) and (2.4.4)and Theorem 2.6.4]{P01}
\cite[equations (2.4.3) and (2.4.4) and Theorem 2.6.4]{Pan2001}.)
\begin{enumerate}
\item[(i)] Multiplication of an $m\times m$ Toeplitz matrix by a vector can be reduced 
to multiplication of two univariate polynomials of degrees $3m-3$ and $m-1$. 
\item[(ii)] Multiplication of an $f$-circulant matrix of size $m\times m$ by a vector can be reduced 
to performing three Fourier transforms at $m$ points and to four multiplications of 
 vectors, one by a scalar $1/m$ and three other ones 
by three diagonal matrices of size $m\times m$. Two of these matrices
 turn into the identity matrix $I_m$  if $f=1$.
 \end{enumerate}
  \end{theorem}

%------------------------------------------------------------------------------
  
\subsubsection{\bf The complexity of Newton's iteration for Toeplitz-like matrices.}\label{stlinv1}
  	
%------------------------------------------------------------------------------

We immediately verify the following result.

\begin{theorem}\label{thinv} (See % \cite[Theorem 1.5.3)]{P01}
\cite[Theorem 1.5.3]{Pan2001}.)
Let $M$ be a nonsingular $m\times m$ matrix. Then
$VM^{-1}-M^{-1}U=-M^{-1}(UM-MV)M^{-1}$ for any pair of  $m\times m$
operator matrices $U$ and $V$, and so if $UM-MV=GH^T$,
then $VM^{-1}-M^{-1}U=G_-H_-^T$ where $G_-=-M^{-1}G$ and $H_-^T=H^TM^{-1}$.
\end{theorem}
Substitute  $U=Z_e$ and $V=Z_f$ and deduce that 
the inversion of a nonsingular matrix given with its  displacement
 does not change the length of this generator (and consequently  does not change 
the  displacement rank), provided that we reverse the order for the pair of operator matrices $(Z_e,Z_f)$,
replacing it by the pair $(Z_f,Z_e)$. 
This observation motivates our search for a short displacement generator 
of the inverse of a Toeplitz-like matrix. Having such generator available, 
we can readily compute the solution ${\bf y}=T^{-1}{\bf f}$ 
of the linear system $T{\bf y}={\bf f}$ by applying the algorithms that support
Theorem \ref{thdsp}.

We can assume that  short displacement generators for the input matrix and
for $X_0=I_m$ are given (note that $Z_eI-IZ_f=Z_e-Z_f=(e-f){\bf i}_1{\bf i}_m^T$
where ${\bf i}_h$ is the $h$th coordinate vector),
and recursively compute short displacement generators of the matrices 
$X_{i+1}=A^{-1}\mod \lambda^{2^{i+1}}$ for $i=0, 1,\dots,p-1$ by applying the following result.

\begin{theorem}\label{thnwt} 
Assume that  we are given a nonsingular $m\times m$ matrix $M$, the matrices $X_i$, $i=0,1,\dots,m$
of Algorithm \ref{algnwt},
and 
any pair of  $m\times m$
operator matrices $U$ and $V$.  Let
 $UM-MV=GH^T$ for some $m\times r$ matrices $G$ and $H$. Then, 
$VX_i-X_iU=G_iH_i^T$ where $G_i=-X_iG$ and $H_i^T=H^TX_i$
for all $i$.
\end{theorem}

\begin{proof}
The theorem follows from Theorem \ref{thinv}
applied  modulo $\lambda^{2^i}$
to the matrices $M=A$ and $M^{-1}=A^{-1}=X_i$.
\end{proof}

We assume $m\times r$  matrices $G$ and $H$ of a  displacement generator for
the input matrix  $A$, and so for every $i$, multiplication of each of these
matrices by  the matrix $X_i=X_{i-1}(2I-AX_{i-1})$
amounts to concurrent multiplication of the matrix by $r$ vectors. 
To operate with the matrices $X_{i-1}$ we recursively define their
displacement generators of length at most $r$  and then
employ the algorithms supporting Theorem \ref{thtmpm}. It follows that
for every $i$, $i=1,\dots,p$
 we compute a short generator  of the matrix $X_i$ at
the computational cost 
dominated by the cost of performing $O(r^2)$ multiplications of bivariate polynomials of degrees 
at most $2m$ in both variables. (We can replace these operations by performing $O(r^2)$ times
two-dimensional Fourier Transform at $m$ points in each dimension. We also need to perform some multiplications of 
vectors by diagonal matrices and $2r-2$ subtractions of vector polynomials
at the dominated computational cost.)
At all $p$ stages of Newton's iteration we perform $O(pr^2)$
multiplications of bivariate polynomials of degrees 
at most $2m$ in both variables.

\begin{corollary}\label{cor:general}  Let $A=I_m-\lambda B$  an $m\times m$ matrix, and assume that $Z_eA-AZ_f=GH^T$ for some $m\times r$ matrices $G$ and $H$. Then the powers $B^0,B,\ldots, B^m$ can be computed via Algorithm \ref{algnwt} $(k=m)$ in parallel complexity $$O_A(\log(m)^2r^2, m^2r^2/\log(m)).$$
\end{corollary}
%------------------------------------------------------------------------------
  
\subsubsection{\bf The overall complexity of solving a Toeplitz-like linear system of equations.}\label{stlinv2}
  	
%------------------------------------------------------------------------------

It remains to estimate the complexity of performing Algorithm \ref{algnul}.

At Stage 1, we compute the trace of the matrix polynomial $A^{-1}\mod \lambda ^m$
expressed via its short displacement generator by using Theorem \ref{thdsp}.
We compute (modulo $\lambda ^m$)  inner products of $m$ pairs of vector polynomials of
dimension $m$ and then sum the $m$ computed values. Clearly, the overall computational 
cost of this operation is dominated by the estimated cost of performing Algorithm \ref{algnwt}. 

Next, recall that the trace of the matrix $B^k$ is the $k$th power sum of the eigenvalues
of the matrix $B$, which are the roots of its characteristic polynomial $c_B$.
At Stage 1, we produce these traces, equal to the power sums, and 
 at Stage 2, we recover the coefficients of the  characteristic polynomial 
from the power sums. We
apply the solution algorithm for Problem 4.8 on pages 34--35 of \cite{BiniPan1994}
which amounts to performing $O(\log m)$ multiplications of polynomials of degrees at most $m$.
Clearly the cost of performing this stage is even stronger dominated.

Stage 3 is reduced essentially to multiplication of $e$-  and $f$-circulant matrices by 
$2m$ vectors. The cost of performing this stage is
dominated by virtue of Theorem \ref{thtmpm}.

\begin{remark}\label{resng} 
Appendix B of \cite{Pan1992} extends 
the expression $B^{-1}=-\sum_{j=0}^{m}(c_i/c_0)B^{i-1}$ used at Stage 3 of
Algorithm \ref{algnul} to express 
the Moore--Penrose generalized inverse of a matrix 
through its characteristic polynomial. 
By employing this expression we can follow the paper \cite{Pan1992}
and readily extends the algorithms and the complexity estimates to the task of computing 
%at the same computational cost of 
the least squares solution of a singular Toeplitz-like linear system of equations.
\end{remark}

\subsubsection{Parallel complexity of the iteration in Definition \ref{constr3}.}

Using the results of this section,  we have the following corollary for the parallel complexity of modular inverse computation:

\begin{corollary}\label{cor:inverse} Let $q(T)\in \K(T)$ be  degree $d$ and $p(T)\in \K(T)$ be degree at most $d-1$ that is relatively prime to $q(T)$. Then, we can compute $p^{-1}(T) \mod q(T)$ in parallel complexity $O_A(\log^2(d), d^2/\log(d))$.\end{corollary} 

\begin{proof} 
It follows from Corollary \ref{cor:general} and the previous subsection and from the fact that the Sylvester matrix of $p$ and $q$ has size $m\times m$ with $m\leq 2d-1$ and displacement rank $r\leq 2$.
\end{proof}

Besides modular inverses,  the computation of the polynomials in  Definition \ref{constr3} is  dominated by the computation of the adjoint of the polynomial matrix $J_\bfF(\bfv (T))$ modulo $q(T)$. We can assume that all polynomials in the polynomial arithmetics involved, as well as our input polynomials in $\bfF$, have degree at most $2d$. Then, according  to \cite[page 311]{BiniPan1994}, the parallel complexity of division with remainder using degrees at most $2d$ and $d$ polynomials is $O_A(\log(d)\log^*(d),d/\log^*(d))$. 
  Moreover,  using  \cite[page 319]{BiniPan1994}, we can compute the adjoint (and the inverse) of an   $n\times n$ scalar matrix in $O_A(\log^2(n),n^{\omega+1})$. Thus, the adjoint of $J_\bfF(\bfv (T))$ modulo $q(T)$ can be computed in 
%!! $O_A(n^\omega\log(d) \log^*(nd), n^2d)$. 
$$O_A(\log^2(n)\log(d) \log^*(d), n^{\omega+1}d/\log^*(d)).$$ 
  
  Combining all the above we get the following proposition. Note that the most significant difference between its 
 complexity bounds 
and the ones in Proposition~\ref{compl1} 
is the extra $d$ factor in the required number of processors.

 \begin{proposition}
Assume that  we are given $\bfF$, $u$, $q(T)$ and $\vv(T)$ satisfying Assumption \ref{assu3}. Assume further that the polynomials in $\bfF$ have degree at most~$2d$. Then we can compute the polynomials   $\bar{Q}(T), \bar{\bfV}(T)=(\bar{V}_1(T), \ldots, 
\bar{V}_n(T))$  of Definition \ref{constr3}
 with the cost  $$O_A( \log^2(n)\log^2(d)\log^*(d) , n^{\omega+1}d^2/\log^*(d)).$$ \end{proposition}
 
 %------------------------------------------------------------------------------

\subsubsection{\bf Simplified computation of cofactors from roots.}\label{stlinv3}
  	
%------------------------------------------------------------------------------

Given two coprime univariate polynomials $u=u(x)$ of a degree $d$ and $v=v(x)$ of a degree $m$
we seek their cofactors $s=s(x)$ of a degree at most $m-1$ and $t=t(x)$ of a degree at most $d-1$
such that $su+tv=1$. This task amounts to the solution of a Sylvester linear system 
of equations, which we can compute by applying the algorithms of the previous subsections.

Next, we consider the case where we are given $d$ distinct roots $x_1,\dots,x_d$ of the polynomial $u$.
(We can proceed similarly where we are given $n$ distinct roots $y_1,\dots,y_m$ of the polynomial $v$.) 
Then we can devise more
efficient algorithms by applying the evaluation/interpolation techniques of \cite{Toom63} as follows.

%------------------------------------------------------------------------------

\begin{algorithm}\label{algCofactor} {\rm Computation of cofactors.}

%------------------------------------------------------------------------------

\begin{description}

%------------------------------------------------------------------------------

\item[{\sc Input:}]  two coprime univariate polynomials $u=u(x)$ of  degree $d$ and $v=v(x)$ of  degree $m\leq d$
and $d$ distinct roots $y_1,\dots,y_d$ of the polynomial $u$. 

%------------------------------------------------------------------------------

\item[{\sc Output:}] two  cofactors, $s=s(x)$ of a degree at most $m-1$ and $t=t(x)$ of a degree at most $d-1$
such that $su+tv=1$.

%------------------------------------------------------------------------------

%\item[{\sc Initialization:}] Set $X_0\leftarrow I$, $A\leftarrow I-\lambda B$  
% for a scalar parameter $\lambda $, $d\leftarrow \lceil \log_2(k+1)\rceil$. 

%------------------------------------------------------------------------------

\item[{\sc Computations:}] $~$

%------------------------------------------------------------------------------
%------------------------------------------------------------------------------

\begin{enumerate}
  \item %1
  Compute the values $t(y_i)=1/v(y_i)$ for $i=1,\dots, d$.
\item %2
Interpolate the polynomial $t(x)$.
\item %3
Apply FFT to evaluate the polynomials $1-tv$ and $u$ and their ratio 
$(1-tv)/u$
at the $2^k$th roots of unity
for $k=\lceil\log_2(d+m)\rceil$. 
\item %4
Apply the inverse FFT to interpolate the polynomial 
$s=(1-tv)/u$.
\end{enumerate}

%------------------------------------------------------------------------------

\end{description}

%------------------------------------------------------------------------------

\end{algorithm}
   
At Stages 1  and 2 we apply the efficient known algorithms (see %\cite[Sections 3.1, 3.3, and 3.13]{P01}
Sections 3.1, 3.3, and 3.13 of \cite{Pan2001}) for the solution of both problems
of multipoint evaluation and interpolation. These algorithms  
essentially amount to 
$O(\log (d))$ concurrent  multiplications and divisions of univariate polynomials of degree at most $d$.
The overall cost of performing these operations 
is substantially smaller than the cost of performing the algorithms of the previous subsections for 
the Sylvester 
Toeplitz-like linear systems of equations,
and surely so is the cost 
of the application of FFT and inverse FFT at Stages 3 and 4 as well. Summarizing we perform the computations of the algorithm
in $O(\log^2(d)\log^*(d),d/\log^*(d))$.
 
%------------------------------------------------------------------------------

%\begin{thebibliography}{\hspace{1cm}}

%\bibitem{c91}

%\end{thebibliography}

%\bibliographystyle{abbrv}
%\bibliography{bibl}

\begin{thebibliography}{10}

\bibitem{AFT2008}
J.~Abbott, C.~Fassino, and M.-L. Torrente.
\newblock Stable border bases for ideals of points.
\newblock {\em J. Symbolic Comput.}, 43(12):883--894, 2008.

\bibitem{AkogHausSza13}
T.A. Akoglu, J.D. Hauenstein, and A.~Szanto.
\newblock Certifying solutions to overdetermined and singular polynomial
  systems over {Q}.
\newblock manuscript, 2013.

\bibitem{AveKrickPac2006}
M.~Avenda{\~n}o, T.~Krick, and A.~Pacetti.
\newblock Newton-{H}ensel interpolation lifting.
\newblock {\em Found. Comput. Math.}, 6(1):81--120, 2006.

\bibitem{BHSW06} D.J. Bates, J.D. Hauenstein, A.J. Sommese, and C.W. Wampler.
\newblock Bertini: software for numerical algebraic geometry.
\newblock Available at \url{bertini.nd.edu}.

\bibitem{BiniPan1994}
D.~Bini and V.~Y. Pan.
\newblock {\em Polynomial and matrix computations. {V}ol. 1}.
\newblock Progress in Theoretical Computer Science. Birkh\"auser Boston Inc.,
  Boston, MA, 1994.
\newblock Fundamental algorithms.

\bibitem{BlCuShSm}
L.~Blum, F.~Cucker, M.~Shub, and S.~Smale.
\newblock {\em Complexity and real computation}.
\newblock Springer-Verlag, New York, 1998.
\newblock With a foreword by Richard M. Karp.

\bibitem{CaPaHaMo2001}
D.~Castro, L.~M. Pardo, K.~H{\"a}gele, and J.~E. Morais.
\newblock Kronecker's and {N}ewton's approaches to solving: a first comparison.
\newblock {\em J. Complexity}, 17(1):212--303, 2001.

\bibitem{Chistov84}
A.~L. Chistov.
\newblock An algorithm of polynomial complexity for factoring polynomials, and
  determination of the components of a variety in a subexponential time.
\newblock {\em Zap. Nauchn. Sem. Leningrad. Otdel. Mat. Inst. Steklov. (LOMI)},
  137:124--188, 1984.
\newblock Theory of the complexity of computations, II.

\bibitem{Fassino2010}
C.~Fassino.
\newblock Almost vanishing polynomials for sets of limited precision points.
\newblock {\em J. Symbolic Comput.}, 45(1):19--37, 2010.

\bibitem{FasTor2013}
C.~Fassino and M.-L. Torrente.
\newblock Simple varieties for limited precision points.
\newblock {\em Theoret. Comput. Sci.}, 479:174--186, 2013.

\bibitem{FerBaiArn1999}
H.~R.~P. Ferguson, D.~H. Bailey, and S.~Arno.
\newblock Analysis of {PSLQ}, an integer relation finding algorithm.
\newblock {\em Math. Comp.}, 68(225):351--369, 1999.

\bibitem{GioJeaVil2003}
P.~Giorgi, C.~Jeannerod, and G.~Villard.
\newblock On the complexity of polynomial matrix computations.
\newblock In {\em {Proceedings of the 2003 International Symposium on Symbolic
  and Algebraic Computation}}, pages 135--142. ACM Press, 2003.

\bibitem{GHHMPM97}
M.~Giusti, J.~Heintz, K.~H{\"a}gele, J.~E. Morais, L.~M. Pardo, and J.~L.
  Monta{\~n}a.
\newblock Lower bounds for {D}iophantine approximations.
\newblock {\em J. Pure Appl. Algebra}, 117/118:277--317, 1997.
\newblock Algorithms for algebra (Eindhoven, 1996).

\bibitem{GHMM98}
M.~Giusti, J.~Heintz, J.~E. Morais, J.~Morgenstern, and L.~M. Pardo.
\newblock Straight-line programs in geometric elimination theory.
\newblock {\em J. Pure Appl. Algebra}, 124(1-3):101--146, 1998.

\bibitem{Grobner-free}
M.~Giusti, G.~Lecerf, and B.~Salvy.
\newblock A {G}r\"obner free alternative for polynomial system solving.
\newblock {\em J. Complexity}, 17(1):154--211, 2001.

\bibitem{Grigorev84}
D.~Y. Grigor{\cprime}ev.
\newblock Factoring polynomials over a finite field and solution of systems of
  algebraic equations.
\newblock {\em Zap. Nauchn. Sem. Leningrad. Otdel. Mat. Inst. Steklov. (LOMI)},
  137:20--79, 1984.
\newblock Theory of the complexity of computations, II.

\bibitem{H13} J.D. Hauenstein.
\newblock Numerically computing real points on algebraic sets.
\newblock {\em Acta Appl. Math.}, 125(1), 105--119, 2013.


\bibitem{Hauenstein-Sotille}
J.~D. Hauenstein and F.~Sottile.
\newblock Algorithm 921: alpha{C}ertified: certifying solutions to polynomial
  systems.
\newblock {\em ACM Trans. Math. Software}, 38(4):Art. ID 28, 20, 2012.

\bibitem{HKPSW2000}
J.~Heintz, T.~Krick, S.~Puddu, J.~Sabia, and A.~Waissbein.
\newblock Deformation techniques for efficient polynomial equation solving.
\newblock {\em J. Complexity}, 16(1):70--109, 2000.

\bibitem{HKPP2009}
D.~Heldt, M.~Kreuzer, S.~Pokutta, and H.~Poulisse.
\newblock Approximate computation of zero-dimensional polynomial ideals.
\newblock {\em J. Symbolic Comput.}, 44(11):1566--1591, 2009.

\bibitem{JKSS04}
G.~Jeronimo, T.~Krick, J.~Sabia, and M.~Sombra.
\newblock The computational complexity of the {C}how form.
\newblock {\em Found. Comput. Math.}, 4(1):41--117, 2004.

\bibitem{Jeronimo-Perrucci}
G.~Jeronimo and D.~Perrucci.
\newblock On the minimum of a positive polynomial over the standard simplex.
\newblock {\em J. Symbolic Comput.}, 45(4):434--442, 2010.

\bibitem{Kaltofen85}
E.~Kaltofen.
\newblock Polynomial-time reductions from multivariate to bi- and univariate
  integral polynomial factorization.
\newblock {\em SIAM J. Comput.}, 14(2):469--489, 1985.

\bibitem{Kaltofen85b}
E.~Kaltofen.
\newblock Sparse {H}ensel lifting.
\newblock In {\em E{UROCAL} '85, {V}ol.\ 2 ({L}inz, 1985)}, volume 204 of {\em
  Lecture Notes in Comput. Sci.}, pages 4--17. Springer, Berlin, 1985.

\bibitem{KanLenLov1988}
R.~Kannan, A.~K. Lenstra, and L.~Lov{\'a}sz.
\newblock Polynomial factorization and nonrandomness of bits of algebraic and
  some transcendental numbers.
\newblock {\em Math. Comp.}, 50(181):235--250, 1988.

\bibitem{KarpRam1990}
R.~Karp and V.~Ramachandran.
\newblock Parallel algorithms for shared-memory machines. 
\newblock {\em Handbook of theoretical computer science}, Vol. A, 869Ð941, Elsevier, Amsterdam, 1990. 

\bibitem{LLL82}
A.~K. Lenstra, H.~W. Lenstra, Jr., and L.~Lov{\'a}sz.
\newblock Factoring polynomials with rational coefficients.
\newblock {\em Math. Ann.}, 261(4):515--534, 1982.

\bibitem{Lich2010}
D.~Lichtblau.
\newblock Exact computation using approximate {G}r\"obner bases.
\newblock Available in the Wolfram electronic library,
 2010.

\bibitem{MoTr2008}
B.~Mourrain and P.~Tr{\'e}buchet.
\newblock Stable normal forms for polynomial system solving.
\newblock {\em Theoret. Comput. Sci.}, 409(2):229--240, 2008.

\bibitem{Pan1992}
V.Y.~Pan.
\newblock Parametrization of {N}ewton's iteration for computations with
  structured matrices and applications.
\newblock {\em Comput. Math. Appl.}, 24(3):61--75, 1992.

\bibitem{Pan2001}
V. Y. Pan,
\newblock {\em Structured Matrices and Polynomials: Unified Superfast Algorithms},
\newblock Birkh\"auser/Springer, Boston/New York, 2001.


\bibitem{Rou99}
F.~Rouillier.
\newblock Solving zero-dimensional systems through the rational univariate
  representation.
\newblock {\em Journal of Applicable Algebra in Engineering, Communication and
  Computing}, 9(5):433--461, 1999.

\bibitem{Shir93}
K.~Shirayanagi.
\newblock An algorithm to compute floating point {G}roebner bases.
\newblock In {\em Mathematical computation with {M}aple {V}: ideas and
  applications ({A}nn {A}rbor, {MI}, 1993)}, pages 95--106. Birkh\"auser
  Boston, Boston, MA, 1993.

\bibitem{Shir96}
K.~Shirayanagi.
\newblock Floating point {G}r\"obner bases.
\newblock {\em Math. Comput. Simulation}, 42(4-6):509--528, 1996.
\newblock Symbolic computation, new trends and developments (Lille, 1993).

\bibitem{Stetter2004}
H.~J. Stetter.
\newblock {\em Numerical polynomial algebra}.
\newblock Society for Industrial and Applied Mathematics (SIAM), Philadelphia,
  PA, 2004.
  
\bibitem{Toom63}
A.~L. Toom.
\newblock The Complexity of a Scheme of Functional Elements Realizing the Multiplication of Integers,
\newblock {\em Soviet Mathematics Doklady}, {\bf 3}, 714--716, 1963.


\bibitem{TraZa2002}
C.~Traverso and A.~Zanoni.
\newblock Numerical stability and stabilization of {G}roebner basis
  computation.
\newblock In {\em Proceedings of the 2002 {I}nternational {S}ymposium on
  {S}ymbolic and {A}lgebraic {C}omputation}, pages 262--269 (electronic). ACM,
  New York, 2002.

\bibitem{Trinks85}
W.~Trinks.
\newblock On improving approximate results of {B}uchberger's algorithm by
  newton's method.
\newblock In B.~Caviness, editor, {\em EUROCAL '85}, volume 204 of {\em Lecture
  Notes in Computer Science}, pages 608--612. Springer Berlin Heidelberg, 1985.

\bibitem{WHS11} C.W. Wampler, J.D. Hauenstein, and A.J. Sommese.
\newblock Mechanism mobility and a local dimension test.
\newblock {\em Mech. Mach. Theory}, 46(9), 1193--1206, 2011.

\bibitem{WLE07} C.W. Wampler, B. Larson, and A. Edrman.
\newblock A new mobility formula for spatial mechanisms.
\newblock In {\em Proc. DETC/Mechanisms \& Robotics Conf., Sept. 4--7, Las Vegas, NV (CDROM)}, 2007.


\bibitem{Winkler88}
F.~Winkler.
\newblock A {$p$}-adic approach to the computation of {G}r\"obner bases.
\newblock {\em J. Symbolic Comput.}, 6(2-3):287--304, 1988.
\newblock Computational aspects of commutative algebra.

\bibitem{Zassenhaus69}
H.~Zassenhaus.
\newblock On {H}ensel factorization. {I}.
\newblock {\em J. Number Theory}, 1:291--311, 1969.

\end{thebibliography}

\def\cprime{$'$}

\end{document}